\documentclass[10pt]{amsart}

\usepackage[utf8]{inputenc}

\usepackage{amssymb,amsthm,amsmath}
\usepackage{mathrsfs}
\usepackage{enumerate}
\usepackage{graphicx,xcolor}
\usepackage{setspace}
\usepackage[hidelinks]{hyperref}
\usepackage{comment}
\usepackage{bbm}

\usepackage{pgfplots}
\usepgfplotslibrary{fillbetween}
\usetikzlibrary{calc}
\pgfplotsset{compat=1.17}

\usepackage[paper=a4paper, left=1.2in, right=1.2in, top=1.2in, bottom=1.2in]{geometry}
\newcommand{\ls}{\leqslant}
\newcommand{\gr}{\geqslant}

\newcommand{\conv}{{\rm conv}}
\newcommand{\vol}{{\rm Vol}}
\newcommand{\pos}{{\rm pos}}
\newcommand{\Var}{{\rm Var}}
\newcommand{\dist}{{\rm dist}}
\newcommand{\diam}{{\rm diam}}

\usepackage{dsfont}

\numberwithin{equation}{section} 

\newtheorem{theorem}{Theorem}[section] 
\newtheorem{lemma}[theorem]{Lemma}
\newtheorem{corollary}[theorem]{Corollary}
\newtheorem{proposition}[theorem]{Proposition}

\theoremstyle{remark}
\newtheorem{remark}[theorem]{Remark}
\newtheorem{claim}{Claim}

\theoremstyle{definition}

\hypersetup{
    colorlinks,
    linkcolor={teal!50!black},
    citecolor={green!40!black},
    urlcolor={cyan!60!black}
}

\title{ \bf Variance lower bounds for geometric functionals of rotationally invariant log-concave random polytopes}

\author{Minas Pafis and Christos Pandis}

\date{}

\begin{document}

\begin{abstract}\footnotesize
We establish variance lower bounds for the intrinsic volumes and the face numbers of random polytopes, generated by independent samples from rotationally invariant log-concave probability measures on $\mathbb{R}^d$ with full support. Our results extend the corresponding Gaussian lower bounds of B\'ar\'any and Vu and of B\'ar\'any and Th\"ale to this broader class of distributions. Our proof builds on the geometric construction introduced by B\'ar\'any and Vu, for which we develop an alternative treatment based on barycentric coordinates. We combine it with local variance estimates for both intrinsic volumes and the number of faces.
\end{abstract}

\maketitle

\section{Introduction}

Random polytopes play a central role in convex and stochastic geometry. They
provide probabilistic models for the approximation of convex bodies and arise
naturally in asymptotic geometric analysis. For general background on
stochastic geometry and this kind of models, we refer to
\cite{SchneiderWeil2008,WeilWieacker1993}; more focused surveys on random
polytopes can be found in \cite{Barany2008,Hug2013,Reitzner2010}.

Let $\mu$ be a
probability measure on $\mathbb{R}^d$, and $(X_n)_{n \in \mathbb{N}}$ be a sequence of independent
random vectors distributed according to $\mu$. For every $n \geq d+1$, set
\[
    K_n^\mu:=\conv\{X_1,\ldots,X_n\}.
\]
Among the most fundamental geometric characteristics of $K_n^\mu$ are its
intrinsic volumes
\[
    V_\ell(K_n^\mu), \quad 1\leq \ell\leq d,
\]
which include, up to normalizing constants, the mean width, surface area and volume, as well as its face numbers \[f_\ell(K_n^\mu), \quad \ell=0,\ldots,d-1.\] A natural problem is to determine the asymptotic behaviour of these
quantities and, in particular, the order of their fluctuations.

Higher-dimensional variance estimates were developed earlier in the classical
uniform model. Let $K_n$ be the convex hull of $n$ independent points
distributed uniformly in a convex body $K\subseteq\mathbb{R}^d$. For the
Euclidean ball, K\"ufer \cite{Kufer1994} obtained an upper bound of the
correct order for the variance of the volume. Reitzner \cite{Reitzner2003},  using the Efron--Stein jackknife inequality, extended this upper bound to convex bodies with $C^2$ boundary and positive Gaussian curvature. He subsequently obtained the matching lower
bound and also determined the order of the variances of all face numbers
\cite{Reitzner2005}. Thus, for every such convex body,
\[
    \Var\bigl(\operatorname{Vol}_d(K_n)\bigr)
    \asymp
    n^{-\frac{d+3}{d+1}}
\]
and
\[
    \Var\bigl(f_\ell(K_n)\bigr)
    \asymp
    n^{\frac{d-1}{d+1}},
    \qquad 0\leq \ell\leq d-1.
\]
B\'ar\'any, Fodor and V\'igh \cite{Bar-Fod-Vig} later established
the corresponding order for every intrinsic volume:
\[
    \Var\bigl(V_\ell(K_n)\bigr)
    \asymp
    n^{-\frac{d+3}{d+1}},
    \qquad 1\leq\ell\leq d.
\]

The Gaussian model, corresponding to the standard Gaussian probability measure
$\gamma_d$ on $\mathbb{R}^d$, has played a particularly important role in the
development of this theory. The asymptotic behaviour of the expected intrinsic volumes of Gaussian
polytopes has been known for a long time. Affentranger
\cite{Affentranger} considered convex hulls of random points with spherically symmetric distributions and, in particular, obtained asymptotic formulas for the expected intrinsic volumes of Gaussian polytopes. Hug, Munsonius
and Reitzner \cite{Hug-Mun-Rei} subsequently derived asymptotic results for the expected values of several geometric functionals of Gaussian polytopes, including face numbers and volumes of their skeletons. Compared with the expectation, considerably
less was initially known about fluctuations. Hug and Reitzner
\cite{Hug-Rei}, obtained upper bounds for the variances of the intrinsic volumes and face numbers of Gaussian polytopes.

A major step in the study of fluctuations of Gaussian random polytopes was
made by B\'ar\'any and Vu \cite{Bar-Vu}. They established central limit
theorems for the volume and the face numbers of Gaussian polytopes. An
essential ingredient in their argument is a geometric construction of a
large family of well separated simplices placed near a suitably chosen sphere. Combined with a
local fluctuation argument, this construction yields the variance lower
bounds required for their central limit theorems. In particular, they
obtained
\[
    \Var\bigl(V_d(K_n^{\gamma_d})\bigr)
    \gtrsim
    (\log n)^{(d-3)/2}
\]
and, for every $\ell\in\{0,\ldots,d-1\}$,
\[
    \Var\bigl(f_\ell(K_n^{\gamma_d})\bigr)
    \gtrsim
    (\log n)^{(d-1)/2}.
\] Together with the upper bounds of Hug and Reitzner \cite{Hug-Rei}, these
estimates determine the order of the variances:
\[
    \Var\bigl(V_d(K_n^{\gamma_d})\bigr)
    \asymp
    (\log n)^{(d-3)/2},
\]
and
\[
    \Var\bigl(f_\ell(K_n^{\gamma_d})\bigr)
    \asymp
    (\log n)^{(d-1)/2},
    \qquad 0\leq \ell\leq d-1.
\]
The precise asymptotic order of the variance of all intrinsic volumes was
determined later by Calka and Yukich \cite{Cal-Yuk}. Using a scaling-limit approach based on parabolic germ--grain models, they proved that, for every
$\ell\in\{1,\ldots,d\}$,
\[
    \lim_{n\to\infty}
    (\log n)^{(d+3)/2-\ell}
    \,\Var\bigl(V_\ell(K_n^{\gamma_d})\bigr)
    =c_{d,\ell},
\]
where $c_{d,\ell}\in[0,\infty)$. Their argument, however, did not exclude the
possibility that $c_{d,\ell}=0$ for the lower-dimensional intrinsic volumes.
This remaining problem was addressed by B\'ar\'any and Th\"ale
\cite{Bar-Th}. Building on the geometric construction of B\'ar\'any and Vu,
they proved
\[
    \Var\bigl(V_\ell(K_n^{\gamma_d})\bigr)
    \gtrsim
    (\log n)^{\ell-(d+3)/2},
    \quad 1\leq\ell\leq d,
\]
thereby showing that the limiting constants in the variance asymptotics of
Calka and Yukich are strictly positive.

Variance estimates for intrinsic volumes have also been studied for random
polytopes generated by distributions other than the Gaussian measure. Grote \cite{GroteGamma} established precise variance asymptotics for the intrinsic volumes and face numbers of generalized Gamma polytopes (we thank Prof. Th\"ale for bringing this result to our attention). In particular, if $\gamma_{d,\alpha,\beta}$ is the probability measure with density proportional to \[\|x\|_2^\alpha e^{-\|x\|_2^\beta/\beta}\] for some $a>-1$ and $\beta \geq 1$, he showed that \[\Var\bigl(V_\ell(K_n^{\gamma_{d,\alpha,\beta}})\bigr) \asymp ( \log n)^{\frac{4\ell-\beta(d+3)}{2\beta}}, \quad \ell=1,\ldots,d\] and \[\Var\bigl( f_\ell(K_n^{\gamma_{d,\alpha,\beta}})\bigr) \asymp (\log n)^{\frac{d-1}{2}}, \quad \ell=0,\ldots,d-1.\]Another particularly useful class of rotationally invariant random polytopes is
provided by the beta and beta-prime models. Kabluchko, Temesvari and
Th\"ale \cite{Kab-Tem-Tha} considered the densities
\[
    f_{d,\beta}(x)
    =
    c_{d,\beta}(1-\|x\|_2^2)^\beta
    \mathbf{1}_{\{\|x\|_2<1\}},
    \quad \beta>-1,
\]
and
\[
    \widetilde f_{d,\beta}(x)
    =
    \widetilde c_{d,\beta}(1+\|x\|_2^2)^{-\beta},
    \quad \beta> d/2,
\]
corresponding respectively to beta and beta-prime distributions. They
obtained exact expressions for the expected intrinsic volumes and expected
facet numbers of the associated random polytopes. The beta family contains, through suitable limiting procedures, both the spherical and Gaussian models. Notice also that, when $\beta\geq0$, the beta distribution is rotationally
invariant and log-concave, but has compact support. Thus, the beta model is a
natural counterpart to the full-support log-concave setting considered in
the present paper.

Very recently, Fodor and Gr\"unfelder \cite{Fod-Gru} determined the order of
the variances of all intrinsic volumes and face numbers of beta polytopes.
If $K_n^\beta$ is the convex hull of $n$ independent points with beta
density of parameter $\beta>-1$, then, for every
$1\leq\ell\leq d$,
\[
    \Var\bigl(V_\ell(K_n^\beta)\bigr)
    \asymp
    n^{-\frac{d+3}{d+1+2\beta}},
\]
while, for every $\ell\in\{0,\ldots,d-1\}$,
\[
    \Var\bigl(f_\ell(K_n^\beta)\bigr)
    \asymp
    n^{\frac{d-1}{d+1+2\beta}}.
\]
They also established central limit theorems for the intrinsic volumes.

\textbf{Our contribution.} The purpose of the present paper is to extend the variance lower bounds
known for Gaussian polytopes to a substantially broader class of
probability measures. More precisely, we consider rotationally invariant
log-concave probability measures $\mu$ on $\mathbb{R}^d$ satisfying
\[
    {\rm supp}(\mu)=\mathbb{R}^d,
\]
and study the random polytope $K_n^\mu$ generated by $n$ independent points
distributed according to $\mu$. Writing its density as $f_{\mu}$,
we define the density superlevel sets
$$
R_t(\mu)
:=
\bigl\{x\in\mathbb{R}^d:f_\mu(x)\geq e^{-t}f_\mu(0)\bigr\}, \quad t>0.
$$
Rotational invariance implies that
$$
R_t(\mu)=\rho_tB_2^d
$$
for a suitable radius $\rho_t>0$. At the critical level $t_n$ introduced below (see Lemma~\ref{lem:choice-tn}), we set
$$
r_n:=\rho_{t_n},
\quad
h_n:=\rho_{t_n+1}-\rho_{t_n}, \quad a_n:=\sqrt{r_nh_n}
$$
The level $t_n$ is chosen so that a region of radial height $h_n$ and tangential radius $a_n$ near $\partial R_{t_n}(\mu)$ has $\mu$-measure of order $1/n$; equivalently,
$$
 f_\mu(r_n)a_n^{d-1}h_n\asymp \frac{1}{n}.
$$
With the above notation, our principal result is the following.

\begin{theorem}\label{thm:main-general-radial}
Let $\mu$ be a rotationally invariant log-concave probability measure, with ${\rm supp}(\mu)=\mathbb{R}^d$. Then, for every $1 \leq \ell \leq d$ and every $n \geq d+1$
\[
\Var\bigl(V_\ell(K_n^\mu)\bigr)
\geq
C(\mu, d,\ell) \,r_n^{2\ell-(d+3)/2}\,h_n^{(d+3)/2},
\] where $C(\mu,d,\ell)>0$.
\end{theorem}

We also prove variance lower bounds for the face numbers of $K_n^\mu$.

\begin{theorem}\label{thm:number-of faces}
Let $\mu$ be a rotationally invariant log-concave probability measure, with ${\rm supp}(\mu)=\mathbb{R}^d$. Then, for every $0 \leq \ell \leq d-1$ and every $n \geq d+2$
\[
\Var\bigl(f_\ell(K_n^\mu)\bigr)
\geq
C(\mu, d,\ell) \,r_n^{(d-1)/2}\,h_n^{-(d-1)/2},
\] where $C(\mu,d,\ell)>0$.
\end{theorem}

Our results recover the corresponding variance lower bounds for random polytopes generated by the family of rotationally invariant log-concave probability measures with densities proportional to \(e^{-\|x\|_2^p/p}\), \(p\geq 1\). In particular, they recover the lower bounds from Grote's \cite{GroteGamma} sharp estimates for $\gamma_{d,0,p}$, and as a consequence, the ones obtained in \cite{Bar-Th,Bar-Vu}.  
\begin{corollary}\label{cor:classical-distributions}
Let $p\geq1$ and $\mu_p$ be the rotationally invariant probability measure on $\mathbb{R}^d$ with density
$$
f_p(x)=c_{d,p}\exp\left(-\frac{\|x\|_2^p}{p}\right), \quad x\in \mathbb{R}^d.
$$
If $K_n^{(p)}$ is the convex hull of $n$ independent random vectors with distribution $\mu_p$, then for every $ n \geq d+1$
$$
\Var\bigl(V_\ell(K_n^{(p)})\bigr)
\geq C_1(p,d,\ell)\,
(\log n)^{\frac{2\ell}{p}-\frac{d+3}{2}},
\quad 1\leq\ell\leq d
$$
and for every $n \geq d+2$
$$
\Var\bigl(f_\ell(K_n^{(p)})\bigr)
\geq C_2(p,d,\ell)\,
(\log n)^{\frac{d-1}{2}},
\qquad 0\leq \ell\leq d-1,
$$
where $C_1(p,d,\ell),\,C_2(p,d,\ell)>0$.
\end{corollary}

Our proof follows the general strategy introduced by B\'ar\'any and Vu
\cite{Bar-Vu} and subsequently used by B\'ar\'any and Th\"ale
\cite{Bar-Th}. We first revisit the geometric construction of B\'ar\'any
and Vu. Using barycentric coordinates, we give alternative proofs of several
geometric properties of the simplices involved in the construction and
establish separation properties that make the conditional
independence of the local configurations explicit. We then adapt the construction to the rotationally
invariant log-concave setting by expressing the relevant estimates in terms
of the natural scales $r_n, \,h_n$ and $a_n$.

For the intrinsic volumes, we follow the local-variance strategy of
B\'ar\'any and Th\"ale, but use a different geometric realization of the
local fluctuation. After conditioning on the tangential position of the
distinguished point, we compare two possible positions lying on the same
radial fiber and in two separated radial subregions. The corresponding local
simplices are then nested, which allows us to obtain a uniform lower bound
for the difference of their projected $\ell$-dimensional volumes. By
Kubota's formula, this yields the required lower bound for the local
variation of the intrinsic volume. A geometric localization argument
transfers this local change to the whole random polytope, and summing the
resulting conditional variances over the well-separated regions gives the
desired global lower bound.

In adapting the argument of B\'ar\'any and Th\"ale, we encountered some
difficulties with certain geometric assertions used in \cite{Bar-Th},
which do not appear to hold in the form stated there and may be due to
typographical or technical inaccuracies. We therefore replace this part of
their proof by the geometric comparison described above.

For the face numbers, the local nature of the construction allows for a more
direct argument. After conditioning on the points outside a given active
region, the dependence on the remaining local points reduces to the face
numbers of a polytope of the form \[
\conv\{Z_1,Z_2,F\},
\] where $F$ is a fixed $(d-1)$-dimensional simplex. We identify
two local configurations of positive probability: in one, the point $Z_2$ is
hidden inside the simplex $\conv\{Z_1,F\}$, while in the other it produces a single
stacking over one of its facets. The corresponding face numbers differ by a
fixed positive amount depending only on $d$ and the dimension of the faces.
A localization argument shows that this local change agrees with the change
of the face number of the whole polytope. Summing the resulting conditional
variance contributions over the separated active regions yields the desired
lower bound.

\section{Notation and background information}
First, we introduce some basic notation and definitions. We work in $\mathbb{R}^d$, which is equipped with the standard inner product $\langle \cdot, \cdot \rangle$. We denote the corresponding Euclidean norm by $\|\cdot\|_2$, and write $B_2^d$ for the Euclidean unit ball and $S^{d-1}$ for the unit sphere. For $p \geq 1$, we denote by $\|\cdot\|_p$ the $\ell_p$-norm. Volume in $\mathbb{R}^d$ is denoted by $\vol_d$, while $\omega_d$ stands for the volume of $B_2^d$.

A  Borel probability measure $\mu$ on $\mathbb{R}^d$ is rotationally invariant if $\mu(UB)=\mu(B)$, for every Borel $B \subseteq \mathbb{R}^d$ and every orthogonal transformation $U \in O(d)$. If $\mu$ has a density, say $f_\mu$, then $f_\mu$ is radial, i.e. $f_\mu(x)=g(\|x\|_2)$, for some function $g: [0,+\infty) \to [0,+\infty)$.

We say that a Borel probability measure $\mu$ on $\mathbb{R}^d$ is  full-dimensional if $\mu(H)<1$ for every hyperplane $H$ in $\mathbb{R}^d$. In addition,  $\mu$ is called log-concave if it is full-dimensional and
$$\mu(\lambda A+(1-\lambda)B) \gr \mu(A)^{\lambda}\mu(B)^{1-\lambda}$$
for any pair of compact sets $A,B$ in ${\mathbb R}^d$ and any $\lambda \in (0,1)$. Borell \cite{Borell-1974} has proved that, under these assumptions, $\mu $ has a log-concave density $f_{{\mu }}$. Recall that a function $f:\mathbb R^d \rightarrow [0,+\infty)$ is called log-concave if its support $\{f>0\}$ is a convex set in ${\mathbb R}^d$ and the restriction of $\log{f}$ to it is concave. If $f$ is an integrable log-concave function on $\mathbb{R}^d$, there exist constants $A,B>0$ such that $f_\mu(x)\ls Ae^{-B\|x\|_2}$ for all $x\in {\mathbb R}^d$ (see \cite[Lemma~2.2.1]{BGVV-book}).

The Grassmannian of $\ell$-dimensional subspaces of $\mathbb{R}^d$ is denoted by $G_{d,\ell}$. For some $L \in G_{d,\ell}$ we write $P_L$ for the orthogonal projection from $\mathbb{R}^d$ onto $L$, and set \[B_L=B_2^d \cap L, \quad S_L=S^{d-1}\cap L. \] The Grassmannian is equipped with the metric \[d_{G_{d,\ell}}(L,M)=\|P_L-P_M\|, \quad L,M \in G_{d,\ell}\] and  with the Haar probability measure $\nu_{d,\ell}$. 

Throughout the paper we write $a \lesssim b$ for positive quantities $a,b$, to declare that $a \leq c(d,\ell) b$, for some positive constant $c(d,\ell)$ depending only on $d$ and $\ell$. We also use $a \asymp b$ to declare that $c_1(d,\ell)a \leq b \leq c_2(d,\ell)a$, where $c_1(d,\ell),c_2(d,\ell)>0$. When $\ell$ is not relevant in the context, the constants are understood to depend only on $d$. We also write $a \lesssim_{\mathcal{T}} b$ and $a \asymp_{\mathcal{T}} b$ when the implicit constants depend additionally on $\mathcal{T}$.

For $u,v \in \mathbb{R}^d \setminus \{0\}$, we define their angle to be \[\angle(u,v)=\arccos\left( \frac{\langle u,v \rangle}{\|u\|_2  \|v\|_2}\right)\] and by convention we set $\angle(u,0)=0$. If $L \in G_{d,\ell}$ , we define \[\angle(u, L)=\inf_{v \in L\setminus\{0\}} \angle(u,v).\] As for every $v \in L$, we have that $-v \in L$, it is clear  that $\, \angle(u,L) \in [0,\pi/2]$. Observe that if $u \in L$, then $\angle(u,L)=0$, whereas if $u \in L^\perp \setminus\{0\}$, then $\angle(u,L)=\pi/2$. It is also easy to check that \[\angle(u,L)=\min_{v \in S_L} \, \arccos \left( \frac{|\langle u,v \rangle|}{\|u\|_2} \right).\]

If $A \subseteq \mathbb{R}^d$ we write $\mathds{1}_A$ for its indicator function. We denote by ${\rm aff}(A)$ and $\conv(A)$  its affine and convex hull respectively. Moreover, we set \[\pos(A)=\left\{\sum_{j=1}^{m}t_jx_j: m \in \mathbb{N}, \, x_j \in A, \, t_j \geq 0 \right\}\] to be its positive hull, which is a convex set.

A convex body in $\mathbb{R}^d$ is a compact convex subset $K$ with non-empty interior. Its support function is $h_K(y)=\max \{\langle x,y\rangle: x \in K\}$ for $y \in \mathbb{R}^d$. If $0 \in {\rm int} (K)$, then we define the Minkowski functional of $K$, as $p_K(y)=\inf \{ t>0: y \in tK\}$ for every $y \in \mathbb{R}^d$.

The Hausdorff metric $d_H$ on the space of convex bodies in $\mathbb{R}^d$ is defined as \[d_H(K_1,K_2) =\inf\{\delta \geq 0: K_1 \subseteq K_2+\delta B_2^d, \quad K_2 \subseteq K_1+\delta B_2^d\}\] for every convex bodies $K_1,K_2$ in $\mathbb{R}^d$.

For a convex body $K \subseteq \mathbb{R}^d$, Steiner's formula states that for every $t \geq 0$ \[\vol_d(K+tB_2^d)=\sum_{\ell=0}^{d}\omega_{d-\ell} V_\ell(K)t^{d-\ell}.\] The coefficient $V_\ell(K)$ is called the $\ell$-intrinsic volume of $K$. In particular, $V_d(K)=\vol_d(K)$ and $V_0(K)=1$. The functionals $V_\ell(K)$ are continuous with respect to the Hausdorff metric. For $1 \leq \ell \leq d$, Kubota's formula gives 
\begin{equation}\label{eq:intrinsic-kubota-formula}  
V_\ell(K)=\binom{d}{\ell} \frac{\omega_d}{\omega_\ell \omega_{d-\ell}}\int_{G_{d,\ell}} \vol_\ell(P_LK) \, d\nu_{d,\ell}(L).
\end{equation}
Therefore, if $K_1 \subseteq K_2$ are convex bodies in $\mathbb{R}^d$, then $V_\ell(K_1) \leq V_\ell(K_2)$. It is also known that for every $1 \leq \ell \leq d$ the functional $V_\ell$ on convex bodies of $\mathbb{R}^d$ is a valuation, i.e. if $K_1, K_2 \subseteq \mathbb{R}^d$ are convex bodies, with $K_1 \cup K_2$ to be also a convex body, then \begin{equation}\label{eq:intrinsic-valuation}
    V_\ell(K_1 \cup K_2) =V_\ell(K_1)+V_\ell(K_2)-V_\ell (K_1 \cap K_2).
\end{equation}

A polytope $P \subseteq \mathbb{R}^d$ is the convex hull of finitely many points. We denote its affine dimension by ${\rm dim}(P)$. A set $F \subseteq P $ is a face of $P$ if there exists a supporting hyperplane $H$ of $P$ such that \[F=P \cap H.\] A face of dimension $\ell$ is called an $\ell$-face, $\ell=0,\ldots,d-1$. The number of $\ell$-faces of $P$ is denoted by $f_\ell(P)$. Thus, \[(f_0(P),\ldots,f_{d-1}(P))\] is the $f$-vector of $P$. If for some $m=1,\ldots,d$, \[P=P_m=\conv\{v_i^0,\ldots,v_i^m\}\] is a $m$-simplex then every face of $P_m$ is the convex hull of a subset of its vertices. Consequently, \begin{equation}\label{eq:m-simplex-faces}
    f_\ell(P_m)=\binom{m+1}{\ell+1}, \quad \ell=0,\ldots,m-1.
\end{equation}
In addition, if $P$ is a $d$-simplex, every $x \in \mathbb{R}^d$ can be written uniquely in barycentric coordinates as \[x=\sum_{j=0}^{d}\lambda_j(x)v_j, \quad \sum_{j=0}^d\lambda_j=1.\] Each $\lambda_j$ is an affine function and the vector \[\lambda(x)=(\lambda_0(x),\ldots,\lambda_d(x))\] is called the barycentric vector of $x$ with respect to $P$ or its vertices. It is easy to see that $x \in P$ is equivalent to \[\lambda(x)\in \Lambda_d=\{(\lambda_0,\ldots,\lambda_d):\lambda_0,\ldots,\lambda_d \geq 0, \, \sum_{j=0}^{d} \lambda_j =1\}  \]

We refer to Schneider's book \cite{Schneider-book} for classical facts from the Brunn-Minkowski theory and to Ziegler's book \cite{Ziegler-book} for polytope theory.

\section{Definition of parameters and auxiliary estimates}

Let \(d\ge 2\), and let \(\mu\) be a rotationally invariant log-concave probability
measure on \(\mathbb R^d\) with unbounded support and density
\[
d\mu(x)=f_\mu(x)\,dx= f_\mu(0)e^{-\phi(\|x\|_2)}\,dx,
\]
where \(\phi:[0,+\infty)\to\mathbb R\)  satisfies
\[
\phi(r)\to+\infty ,\] as $ r\to+\infty$.

Since \(f\) is log-concave, the function
\[
\Phi(x):=\phi(\|x\|_2)
\]
is convex on \(\mathbb R^d\). In particular, by restricting \(\Phi\) to the ray
\(\{re_1:r\ge0\}\), the radial profile \(\phi\) is convex on \([0,+\infty)\). 

Moreover, \(\phi\) is automatically increasing. Indeed, if \(0\le r\le s\), then
\[
re_1=\lambda se_1+(1-\lambda)(-se_1),
\quad
\lambda=\frac{1+r/s}{2}\in[0,1].
\]
By convexity of \(\Phi\) and radiality,
\[
\phi(r)=\Phi(re_1)
\le
\lambda\Phi(se_1)+(1-\lambda)\Phi(-se_1)
=
\phi(s).
\]
Thus \(\phi\) is increasing on \([0,+\infty)\). We note that by the definition of $f_\mu$ we have $\phi(0)=0$, which implies that $\phi$ is continuous on $[0,+\infty)$. Indeed for the latter, if $M>0$ is fixed, then for every $0<x<M$, by convexity \[\phi(x) \leq \frac{x}{M}\,\phi(M).\] Taking limits $x \to 0$, yields the continuity at $0$.

For $t \geq 0$, define \[\rho_t=\sup\{r \geq 0: \phi(r) \leq t\}.\] It is apparent that $t \mapsto \rho_t$ is increasing. Since $\phi$ is continuous, finite and $\phi(r) \to +\infty,$ as $r \to +\infty$, we have \[\phi(\rho_t)=t, \quad t \geq 0.\] By continuity of $\phi$, we can deduce that \[\rho_t \to +\infty, \] as $t \to +\infty.$

\begin{lemma}\label{lem:rho-concave}
    The function \[t \mapsto \rho_t, \quad t \geq0\] is concave.
\end{lemma}

\begin{proof} 
    Let $t,s\geq 0$ and $\lambda \in [0,1]$. By convexity of $\phi$
    \[\phi(\lambda \rho_t+(1-\lambda)\rho_s ) \leq \lambda t+(1-\lambda)s.\] Therefore, by its definition \[\lambda\rho_t+(1-\lambda)\rho_s \leq \rho_{\lambda t+(1-\lambda)s}.\]
\end{proof}

We can also control the increments of $\rho_t$.

\begin{lemma}\label{lem:rho-increments}
For every $t>0$
\[
0 \leq\frac{\rho_{t+1}-\rho_t}{\rho_t} \leq \frac{1}{t}.
\]
\end{lemma}

\begin{proof}
The left-hand side inequality is obvious. So we will only deal with the right-hand side one.
For $t>0$ we write \[t=\frac{t}{t+1}(1+t)+\frac{1}{t+1}\,0.\] So by Lemma~\eqref{lem:rho-concave} \[\rho_t \geq \frac{t}{t+1}\rho_{t+1}+\frac{1}{t+1}\rho_0 \geq \frac{t}{t+1}\rho_{t+1},
\]
 which yields the lemma.
\end{proof}

Now we define the density superlevel sets
\[
R_t(\mu):=\{x:f_\mu(x)\ge e^{-t}f_\mu(0)\}, \quad t>0.
\]
Notice that since \(\mu\) is rotationally invariant,
\[
R_t(\mu)=\rho_tB_2^d,
\]
We also set
\[
q_\mu(r):=\mu(\{x:\langle x,e_1\rangle\ge r\}), \quad r>0.
\]
By rotational invariance, $q_\mu(r)$ is the measure of every half-space that its distance from the origin is $r$ and does not contain the origin. Furthermore,  $q_\mu$ is decreasing.

We now choose a sequence \(t_n\to+\infty\), and set
\begin{equation}\label{eq:r-h-a}
r_n:=\rho_{t_n},
\quad
h_n:=\rho_{t_n+1}-\rho_{t_n},
\quad
a_n:=\sqrt{r_n h_n}.
\end{equation}

\begin{lemma}\label{lem:choice-tn}
There exists a strictly increasing sequence $(t_n)_{n \in \mathbb{N}}$, with \(t_n\to+\infty,\) as $n \to \infty$, such that for every $n \in \mathbb{N}$
\[
\frac{c_{\mu,d}}{n} = f_\mu(r_n)a_n^{d-1}h_n\leq \frac1n,
\] where $c_{\mu,d}>0$ is a constant depending on $\mu$. 
\end{lemma}

\begin{proof}
Define
\[
G(t):=f_\mu(\rho_t)\,
\bigl(\rho_t(\rho_{t+1}-\rho_t)\bigr)^{\frac{d-1}{2}}
(\rho_{t+1}-\rho_t), \quad t>0,
\]
which is continuous, by the continuity of $\rho_t$ and $f_\mu$. We will first check that \(G(t)\to 0\), as $t \to +\infty$.

By Lemma~\ref{lem:rho-increments} for $t \geq 1$, \[\rho_{t+1}-\rho_t\le \rho_t.\] Hence, as $f_\mu(\rho_t)=f_\mu(0)e^{-t}$,
\[
G(t)\le f_\mu(0)e^{-t}\rho_t^d.
\]
Since \(\phi\) is increasing, convex, and tends to \(+\infty\), there exist \(r_\mu>0\) and \(\tilde{c}_\mu>0\) such that
\[
\phi(r)\ge \tilde{c}_\mu r,
\quad r\ge r_\mu.
\]
Therefore \(\rho_t\le C_\mu t\) for all large \(t\), and thus
\[
G(t)\le f_\mu(0)C_\mu^d e^{-t}t^d\to 0, 
\]
as $ t \to +\infty$. Fix \(t_0 \geq 1\) with \(G(t_0)>0\) and set \[c_{\mu,d}:=\min\{1,G(t_0)\}>0.\] We can inductively choose \[1 \leq t_1<t_2<\ldots\] such that $G(t_n)=c_{\mu,d}/n$ for every $n \in \mathbb{N}$. Indeed, we first choose $t_1$. If $G(t_0)=c_{\mu,d}$, then set $t_1=t_0$. Assuming otherwise, since $G(t) \to 0$, as $t \to +\infty$, by continuity of $G$, we can find $t_1 > t_0$, with $G(t_1)=c_{\mu,d}$. Suppose now that $t_n$ has been chosen. Notice that \[\frac{c_{\mu,d}}{n+1}<G(t_n).\] Because $G(t) \to 0$, as $t \to +\infty$, there exists $T>t_n$ such that \[G(T)< \frac{c_{\mu,d}}{n+1}.\] Continuity of $G$ then yields some $t_{n+1} \in (t_n,T)$ with \[G(t_{n+1})=\frac{c_{\mu,d}}{n+1}.\]
\end{proof}

\begin{remark}
    At this point, we discuss the parameters $r_n,a_n$ and $h_n$ associated with the probability measures $\mu_p$ of Corollary~\ref{cor:classical-distributions}. Fix $p \geq 1$ and observe that \[\rho_t=(pt)^{1/p}, \quad t>0.\] Since $\rho_t^\prime=\rho_t^{1-p}$ for every $t>0$, using the mean value theorem we can deduce that \[\frac{h_n}{r_n^{1-p}}\to 1,\] as $n \to \infty$. Hence, \eqref{eq:r-h-a} and Lemma~\ref{lem:choice-tn} imply that \[r_n\asymp_p (\log n)^{1/p}.\] Insert these estimates in Theorem~\ref{thm:main-general-radial} and Theorem~\ref{thm:number-of faces} and we obtain Corollary~\ref{cor:classical-distributions}.
\end{remark}

With the previous definition, we can lower bound the following quantities, independently of $\mu$.
\begin{lemma}\label{lem:relations}
We have that \[\frac{h_n}{a_n}, \ \frac{a_n}{r_n} \lesssim \frac{1}{\sqrt{\log n}}, \quad \frac{h_n}{r_n} \lesssim \frac{1}{\log n}.\]  
\end{lemma}

\begin{proof}
    It suffices to prove the last inequality, because \[\frac{a_n}{r_n}=\frac{h_n}{a_n}=\sqrt\frac{h_n}{r_n}.\]
    To begin with, the definition of $r_n$ in \eqref{eq:r-h-a} and Lemma~\ref{lem:choice-tn} give \begin{equation}\label{eq:rn-definition}
        \frac{c_{\mu,d}}{n}=f_\mu(0)e^{-t_n}r_n^{d}\left(\frac{h_n}{r_n}\right)^{\frac{d+1}{2}}
    \end{equation}
    Convexity of $\phi$ and $\phi(0)=0$ imply that $t \mapsto \phi(t)/t$ is increasing on $(0,+\infty)$. Hence, integrating in polar coordinates and using $\phi(\rho_t)=t$, 
    \begin{align*}
        1&=d\omega_df_\mu(0)\int_{0}^{+\infty} r^{d-1} e^{-\phi(r)}
        \, dr \\
        &\leq d\omega_df_\mu(0) \left[\int_{0}^{\rho_t}r^{d-1} \,dr+\int_{\rho_t}^{+\infty}r^{d-1}e^{-tr/\rho_t}\, dr\right]\\
        &=d\omega_df_\mu(0)\,\rho_t^d\left[\frac{1}{d} +\int_{1}^{+\infty} u^{d-1}e^{-tu} \, du \right] \lesssim f_\mu(0)\rho_t^d,
    \end{align*} if $t \geq 1$ and therefore \[f_\mu(0)\rho_t^d \gtrsim 1, \quad t \geq 1.\] Since by construction $t_n \geq 1$ for every $n \in \mathbb{N}$, applying the last inequality for $t=t_n$, \eqref{eq:rn-definition} yields \[\frac{1}{n}\geq \frac{c_{\mu,d}}{n} \gtrsim e^{-t_n}\left(\frac{h_n}{r_n}\right)^{\frac{d+1}{2}}. \] Thus, \begin{equation}\label{eq:fraction-upper-bound}
    \frac{h_n}{r_n} \lesssim e^{\frac{2t_n}{d+1}}\,n^{-\frac{2}{d+1}}.\end{equation} Now Lemma~\ref{lem:rho-increments} gives \[\frac{h_n}{r_n}\leq \frac{1}{t_n}.\] We split into two cases. If $t_n\geq \frac{\log n}{2}$ then \[\frac{h_n}{r_n} \leq \frac{2}{\log n}.\] If $t_n<\frac{\log n}{2}$, then from \eqref{eq:fraction-upper-bound} \[\frac{h_n}{r_n} \lesssim n^{-\frac{1}{d+1}} \lesssim \frac{1}{\log n}.\]
\end{proof}

We will now estimate the measure of the half-spaces, not containing the origin, and having distance $r_n$ from it.

\begin{proposition}\label{prop:q-shell-comparison}
We have that
\[
q_\mu(r_n)\asymp f_\mu(r_n)a_n^{d-1}h_n.
\]
\end{proposition}

\begin{proof}
Since \(t\mapsto \rho_t\) is concave, for every \(k\ge 0\),
\begin{equation}\label{eq:rho-concavity-step}
\rho_{t_n+k+1}-r_n
\le (k+1)(\rho_{t_n+1}-\rho_{t_n})
=(k+1)h_n.
\end{equation}
We recall the standard cap-volume estimate. If
\[
\Gamma(R,\Delta):=\{x\in\mathbb R^d:\ \|x\|_2\le R,\ x_1\ge R-\Delta\},
\quad
0<\Delta\le \frac R2,
\]
is a cap of height \(\Delta\) on the ball of radius \(R\), then
\begin{equation}\label{eq:cap-volume}
\vol_d(\Gamma(R,\Delta))
\asymp
(R\Delta)^{\frac{d-1}{2}}\Delta.
\end{equation}
Indeed,
\begin{align*}
\vol_d(\Gamma(R,\Delta))
&=
\int_{R-\Delta}^R
\vol_{d-1}\bigl(\Gamma(R,\Delta)\cap\{x_1=s\}\bigr)\,ds  \\
&=
\omega_{d-1}\int_{R-\Delta}^R
(R^2-s^2)^{\frac{d-1}{2}}\,ds.
\end{align*}
For \(s\in[R-\Delta,R]\),
\[
R^2-s^2=(R-s)(R+s)\asymp R(R-s),
\]
which gives \eqref{eq:cap-volume}.

\medskip

\noindent\textbf{Lower bound.}
Let
\[
A_{n,0}:=
\{x:\langle x,e_1\rangle\ge r_n\}
\cap
\bigl(R_{t_n+1}(\mu)\setminus R_{t_n}(\mu)\bigr).
\]
Since
\[
R_{t_n}(\mu)=r_nB_2^d,
\quad
R_{t_n+1}(\mu)=(r_n+h_n)B_2^d,
\]
we have
\[
\vol_d(A_{n,0})
=
\vol_d(\Gamma(r_n+h_n,h_n)).
\]
Because \(h_n/r_n\to0\) by Lemma~\ref{lem:relations}, \eqref{eq:cap-volume} yields
\begin{equation}\label{eq:vol-lower-shell}
\vol_d(A_{n,0})
\asymp
(r_nh_n)^{\frac{d-1}{2}}h_n
=
a_n^{d-1}h_n.
\end{equation}
Moreover, for every \(x\in A_{n,0}\),
\[
e^{-1}f_\mu(r_n)\le f_\mu(x)\le f_\mu(r_n),
\]
because \(x\in R_{t_n+1}(\mu)\setminus R_{t_n}(\mu)\). Hence
\[
\mu(A_{n,0})
\asymp
f_\mu(r_n)\vol_d(A_{n,0})
\asymp
f_\mu(r_n)a_n^{d-1}h_n.
\]
Since \(A_{n,0}\subseteq \{x:\langle x,e_1\rangle\ge r_n\}\), we obtain
\begin{equation}\label{eq:q-lower}
q_\mu(r_n)\gtrsim f_\mu(r_n)a_n^{d-1}h_n.
\end{equation}

\noindent\textbf{Upper bound.}
For \(k\ge0\), define
\[
A_{n,k}:=
\{x:\langle x,e_1\rangle\ge r_n\}
\cap
\bigl(R_{t_n+k+1}(\mu)\setminus R_{t_n+k}(\mu)\bigr).
\]
Then
\[
\{x:\langle x,e_1\rangle\ge r_n\}
=
\bigcup_{k=0}^\infty A_{n,k}
\]
as a disjoint union and therefore
\[
q_\mu(r_n)=\sum_{k=0}^\infty \mu(A_{n,k}).
\]
For \(x\in A_{n,k}\), we have
\begin{equation}\label{eq:density-upper-k}
f_\mu(x)\le e^{-k}f_\mu(r_n).
\end{equation}
Also, by \eqref{eq:rho-concavity-step},
\begin{equation}\label{eq:rho-upper-k}
\rho_{t_n+k+1}-r_n\le (k+1)h_n.
\end{equation}
Hence
\[
A_{n,k}
\subseteq
\Gamma\bigl(\rho_{t_n+k+1},\,\rho_{t_n+k+1}-r_n\bigr).
\]
Let
\[
K_n:=\left\lfloor \frac{r_n}{4h_n}\right\rfloor.
\]
If \(0\le k\le K_n-1\), then, 
\[
\rho_{t_n+k+1}-r_n\le (k+1)h_n\le \frac{r_n}{4}.
\]
Therefore \eqref{eq:cap-volume} gives
\[
\vol_d(A_{n,k})
\lesssim
\bigl(r_n(k+1)h_n\bigr)^{\frac{d-1}{2}}(k+1)h_n
=
(k+1)^{\frac{d+1}{2}}a_n^{d-1}h_n.
\]
Together with \eqref{eq:density-upper-k},
\begin{equation}\label{eq:upper-small-k}
\mu(A_{n,k})
\lesssim
e^{-k}(k+1)^{\frac{d+1}{2}}
f_\mu(r_n)a_n^{d-1}h_n.
\end{equation}
Thus
\[
\sum_{k=0}^{K_n-1}\mu(A_{n,k})
\lesssim
f_\mu(r_n)a_n^{d-1}h_n,
\]
because
\[
\sum_{k=0}^\infty e^{-k}(k+1)^{\frac{d+1}{2}}<+\infty.
\]
If \(k \geq K_n\), then
\[
A_{n,k}\subseteq R_{t_n+k+1}(\mu).
\]
By \eqref{eq:rho-upper-k},
\[
\rho_{t_n+k+1}\le r_n+(k+1)h_n\leq 5(k+1)h_n
\]
for all \(k \geq K_n\). Hence
\[
\vol_d(A_{n,k})
\lesssim
(k+1)^d h_n^d.
\]
Again by \eqref{eq:density-upper-k},
\begin{equation}\label{eq:upper-large-k}
\mu(A_{n,k})
\lesssim
e^{-k}(k+1)^d f_\mu(r_n)h_n^d.
\end{equation}
Therefore
\[
\sum_{k=K_n}^{\infty}\mu(A_{n,k})
\lesssim
f_\mu(r_n)h_n^d
\sum_{k=K_n}^{\infty} e^{-k}(k+1)^d.
\]
The tail satisfies
\[
\sum_{k=K_n}^{\infty} e^{-k}(k+1)^d
\lesssim e^{-K_n/2}.
\]
 Hence
\[
\frac{
f_\mu(r_n)h_n^d
\sum_{k=K_n}^{\infty} e^{-k}(k+1)^d
}{
f_\mu(r_n)r_n^{\frac{d-1}{2}}h_n^{\frac{d+1}{2}}
}
\lesssim
e^{-K_n/2}
\left(\frac{h_n}{r_n}\right)^{\frac{d-1}{2}} \lesssim 1,
\]  due to Lemma~\ref{lem:relations}.
Thus
\[
\sum_{k=K_n}^{\infty}\mu(A_{n,k})
\lesssim f_\mu(r_n)a_n^{d-1}h_n.
\]
Combining the estimates gives
\begin{equation}\label{eq:q-upper}
q_\mu(r_n)\lesssim f_\mu(r_n)a_n^{d-1}h_n.
\end{equation}
Together with \eqref{eq:q-lower}, this proves the claim.
\end{proof}

As an immediate consequence of Proposition~\ref{prop:q-shell-comparison} and the choice of \(t_n\),
\begin{equation}\label{eq:q-one-over-n}
\frac{c_{\mu,d}}{n} \lesssim q_\mu(r_n)\lesssim \frac1n.
\end{equation}

We shall use the following additional shell-regularity assumption along the sequence \(t_n\).

\begin{lemma}\label{lem:inward-q}
For every fixed \(L>0\), there exists $n_0:=n_0(\mu,L) \in \mathbb{N}$ such that for every $n \geq n_0$,
\[
q_\mu(r_n-Lh_n)\lesssim_L f_\mu(r_n)a_n^{d-1}h_n
,\]

\end{lemma}

\begin{proof}
Pick $n_0 \in \mathbb{N}$ so that $t_{n}>L$ and $r_n/h_n>L$ for every $n \geq n_0$. We can find such $n_0$ by Lemma~\ref{lem:choice-tn} and Lemma~\ref{lem:relations}.

We now record a consequence of the concavity of \(t\mapsto\rho_t\). Since \(\rho_t\) is increasing and concave, its increments are decreasing. Hence, for every fixed \(L> 0\) and every $n \geq n_0$,
\[
\frac{\rho_{t_n}-\rho_{t_n-L}}{L}
\ge
\rho_{t_n+1}-\rho_{t_n}
=
h_n.
\]
Therefore,
\begin{equation}\label{eq:rho-backward}
\rho_{t_n-L}\le r_n-Lh_n.
\end{equation}
Since \(f\) is radial and decreasing, if
\[
r_n-Lh_n\le \|x\|_2\le r_n,
\]
then by \eqref{eq:rho-backward},
\[
\|x\|_2\ge \rho_{t_n-L}.
\]
Thus
\begin{equation}\label{eq:density-inner-control}
f_\mu(x)\le f_\mu(\rho_{t_n-L})
=
e^{-(t_n-L)}f_\mu(0)
=
e^L f_\mu(r_n).
\end{equation}
Decompose
\[
q_\mu(r_n-Lh_n)
=
\mu(\{x:x_1\ge r_n-Lh_n,\ \|x\|_2\le r_n\})
+
\mu(\{x:x_1\ge r_n-Lh_n,\ \|x\|_2\ge r_n\}).
\]

The first term is an inner cap of height \(Lh_n\) in \(r_nB_2^d\). Hence, by \eqref{eq:cap-volume} and \eqref{eq:density-inner-control},
\[
\mu(\{x:x_1\ge r_n-Lh_n,\ \|x\|_2\le r_n\})
\lesssim_L
f_\mu(r_n)(r_nh_n)^{\frac{d-1}{2}}h_n
=
f_\mu(r_n)a_n^{d-1}h_n.
\]
For the outer part, define
\[
B_{n,k}:=
\{x:x_1\ge r_n-Lh_n\}
\cap
\bigl(R_{t_n+k+1}(\mu)\setminus R_{t_n+k}(\mu)\bigr),
\quad k\ge0.
\]
Notice that 
\[
\{x: x_1\geq r_n-L h_n, \, \|x\|_2 \geq r_n\} \subseteq \bigcup_{k=0}^{\infty} B_{n,k}.
\]
For \(x\in B_{n,k}\),
\[
f_\mu(x)\le e^{-k}f_\mu(r_n).
\]
Moreover, by concavity of \(t\mapsto\rho_t\),
\[
\rho_{t_n+k+1}-(r_n-Lh_n)
\le
(L+k+1)h_n.
\]
For $0\le k\le K_n-1$, where $K_n:=\lfloor r_n/(4h_n)\rfloor$, the cap estimate gives
\[
\vol_d(B_{n,k})
\lesssim_L
(L+k+1)^{\frac{d+1}{2}}a_n^{d-1}h_n.
\]
Hence
\[
\sum_{k=0}^{K_n-1}\mu(B_{n,k})
\lesssim_L
f_\mu(r_n)a_n^{d-1}h_n
\sum_{k=0}^{\infty}e^{-k}(L+k+1)^{\frac{d+1}{2}}
\lesssim_L
f_\mu(r_n)a_n^{d-1}h_n.
\]
For \(k \geq K_n\), as in the proof of Proposition~\ref{prop:q-shell-comparison},
\[
\vol_d(B_{n,k}) \leq \vol_d(R_{t_{n}+k+1})
\lesssim
(k+1)^dh_n^d,
\]
and hence
\[
\sum_{k=K_n}^{\infty}\mu(B_{n,k})
\lesssim
f_\mu(r_n)h_n^d\sum_{k=K_n}^{\infty}e^{-k}(k+1)^d \lesssim f_\mu(r_n)a_n^{d-1}h_n.
\]
Combining the estimates proves the lemma.
\end{proof}

\section{The geometric construction}

In this section, we present the geometric construction of \cite{Bar-Vu}. The quantities $b_1,b_2$ that appear from now on, are positive constants, depending only on the dimension $d$, and can be chosen as large or as small we want respectively, in order for our arguments to hold true.

We construct the simplices on the sphere
\[
\partial R_{t_n}(\mu)=r_nS^{d-1}.
\]
Let \(y_1,\ldots,y_m\in r_nS^{d-1}\) form a maximal $2b_1a_n$-separated subset of $r_nS^{d-1}$; that is,
\[
\|y_i-y_j\|_2\ge 2b_1 a_n
\quad (i\neq j),
\]
where \(b_1>0\) is a sufficiently large constant. 

Let \(\sigma_{r_n}\) denote the \((d-1)\)-dimensional surface measure on the sphere
$r_nS^{d-1}$. For \(y\in r_nS^{d-1}\) and \(\rho>0\), write
\[
C(y,\rho):=r_nS^{d-1}\cap B(y,\rho)
\]
for the corresponding spherical cap of Euclidean radius \(\rho\).

Since the caps \(C(y_i,b_1a_n)\) are pairwise disjoint and the caps \(C(y_i,2b_1a_n)\) cover \(r_nS^{d-1}\), we have
\begin{equation}\label{eq:cap-packing-covering}
m\,\sigma_{r_n}\bigl(C(y_1,b_1a_n)\bigr)
\le
\sigma_{r_n}(r_nS^{d-1})
\le
m\,\sigma_{r_n}\bigl(C(y_1,2b_1a_n)\bigr).
\end{equation}
Using local coordinates on the sphere,
\begin{equation}\label{eq:small-cap-measure}
\sigma_{r_n}(C(y,\rho))\asymp \rho^{d-1}
\quad\text{uniformly whenever } \rho/r_n\to0.
\end{equation}
Since, by Lemma~\ref{lem:relations}, $a_n/r_n \to 0$, as $ n \to \infty \ $, applying \eqref{eq:small-cap-measure} with \(\rho=b_1a_n\) and \(\rho=2b_1a_n\), we get
\[
\sigma_{r_n}(C(y_1,b_1a_n))\asymp a_n^{d-1},
\quad
\sigma_{r_n}(C(y_1,2b_1a_n))\asymp a_n^{d-1}.
\]
Since also
\[
\sigma_{r_n}(r_nS^{d-1})=d\omega_d r_n^{d-1}\asymp r_n^{d-1},
\]
the inequalities in \eqref{eq:cap-packing-covering} imply
\[
m\,a_n^{d-1}\asymp r_n^{d-1}.
\]
Therefore
\begin{equation}\label{eq:m-estimate}
m\asymp \left(\frac{r_n}{a_n}\right)^{d-1}.
\end{equation}
Let \(u_i\in S^{d-1}\) be such that
\[
y_i=r_nu_i.
\]
Define
\begin{equation}\label{eq:def-yi0} 
y_i^0:=(r_n+h_n)u_i=(1+h_n/r_n)y_i\in \partial R_{t_n+1}(\mu),
\end{equation}
and let
\[
H_i:=\{x:\langle x,u_i\rangle=r_n\}
\]
be the tangent hyperplane to \(r_nS^{d-1}\) at \(y_i\).

Inside \(H_i\), choose a regular \((d-1)\)-simplex centered at \(y_i\), with vertices
\[
y_i^1,\ldots,y_i^d\in H_i\cap S(y_i,ca_n),
\]
where \(0<c<\sqrt{2}\) is a constant. In what follows, we summarize the main properties of the points $(y_i^j)_{1 \leq j \leq d}$.
\begin{proposition}\label{prop:main-properties} 
Let $1 \leq i \leq m$.
\begin{enumerate}
\item[{\rm (i)}] \[y_i=\frac{1}{d}\sum_{j=1}^d y_i^j.\]
\item[{\rm (ii)}] For every $1 \leq j \leq d$, \[y_i^j-y_i \perp y_i.\] In particular, for fixed $i$, we can write \[y_i^j=y_i+w_j=r_nu_i+w_j,\] where \[w_j=y_i^j-y_i \in y_i^\perp, \quad \|w_j\|_2=ca_n.\]
\item[{\rm (iii)}]For every $1 \leq j \leq d$: \[\|y_i^j\|_2^2=\|y_i\|_2^2+\|w_j\|_2^2=r_n^2+c^2a_n=r_n^2+c^2r_nh_n.\] 
Therefore, \[y_i^j \in (r_n+h_n)B_2^d \setminus r_nB_2^d.\]
\item[{\rm (iv)}] For every $1 \leq j,k \leq d$: \[\langle y_i^j,y_i^k-y_i\rangle=\langle y_i^j-y_i,y_i^k-y_i\rangle= \langle w_j,w_k \rangle= \begin{cases}
c^2a_n^2 &, \text{ if } \, k=j\\
-\dfrac{c^2a_n^2}{d-1} &, \text{ if } \, k \neq j
\end{cases}.\]
\item[{\rm (v)}] For every $1\leq j \neq k \leq d$:  \[\|y_i^j-y_i^k\|_2=ca_n\sqrt{\frac{2d}{d-1}}.\]
\item[{\rm (vi)}] \[u_i^\perp=y_i^\perp={\rm span}\{y_i^j-y_i:1 \leq j \leq d \}={\rm span}\{w_j:1 \leq j \leq d \}.\]
\item[{\rm (vii)}] The vectors $w_1,\ldots,w_d$ are vertices of a regular $(d-1)$-simplex centered at $0$. Moreover, \[\frac{ca_n}{d-1}B_{u_i^\perp} \subseteq \conv\{w_1,\ldots,w_d\}.\]

\end{enumerate}
    
\end{proposition}

Define now
\[
\Delta_i:=\conv\{y_i^0,y_i^1,\ldots,y_i^d\}.
\]
It is clear by construction that $y_i^0,\ldots,y_i^d$ are affinely independent. Hence, $\Delta_i$ is a $d$-simplex. We denote by $\lambda(x)$ the barycentric vector of some $x \in \mathbb{R}^d$, with respect to $\Delta_i$, when $i$ is fixed.
For \(j=0,1,\ldots,d\), let also
\[
\Delta_i^j:=y_i^j+b_2(\Delta_i-y_i^j)=(1-b_2)y_i^j+b_2\Delta_i,
\]
where \(0<b_2<1/2\) is sufficiently small. Observe  that since by Lemma~\ref{lem:relations} \[\|y_i^0-y_i\|_2=h_n \lesssim a_n,\] using also Proposition~\ref{prop:main-properties} ${\rm (v)}$, we obtain \begin{equation}\label{eq:simplices-diam}
\diam(\Delta_i) \asymp a_n, \ \diam( \Delta_i^j) \asymp b_2a_n, \quad 0\leq j \leq d.
\end{equation}
Moreover, \begin{equation}\label{eq:simplices-volume-asymptotics}
\vol_d(\Delta_i)\asymp a_n^{d-1}h_n, \ \vol_d(\Delta_i^j)\asymp b_2 a_n^{d-1}h_n , \quad 0\leq j \leq d.
\end{equation}

\begin{lemma}\label{lem:simplex-shell}
For every $i=1,\ldots,m$ and every $j=0,\ldots, d$ we have that
\[
\Delta_i^j \subseteq\Delta_i\subseteq R_{t_n+1}(\mu)\setminus {\rm int}(R_{t_n}(\mu)).
\]
\end{lemma}

\begin{proof}
It is enough to show the lemma for $\Delta_i$, as it is obvious that $\Delta_i^j \subseteq \Delta_i$. Recall that \[\Delta_i=\conv\{y_i^0,\ldots,y_i^d\}.\] Considering \eqref{eq:def-yi0} and Proposition~\ref{prop:main-properties} $\rm{(iii)}$, if $x$ is a vertex of $\Delta_i$, we can easily deduce that \[\|x\|_2 \leq r_n+h_n,\] as $c<\sqrt{2}$. Hence, by convexity $\Delta_i \subseteq R_{t_n+1}(\mu)$. Let now $x \in \Delta_i$ and write \[x=\sum_{j=0}^{d} \lambda_j(x) y_i^j,\] as a convex combination. Using again \eqref{eq:def-yi0} and Proposition~\ref{prop:main-properties} $\rm{(ii)}$, we get that \[x=(r_n+\lambda_0(x)h_n)u_i+\sum_{j=1}^d \lambda_j(x)w_j,\] where $w_j \perp u_i$ for every $1 \leq j \leq d$. Therefore, \[\|x\|_2^2=(r_n+\lambda_0(x)h_n)^2+\biggl\|\sum_{j=1}^d \lambda_j(x)w_j\biggr\|_2^2 \geq r_n^2.\] 
\end{proof}

\begin{lemma}\label{lem:simplex-mass}
One has
\[
\mu(\Delta_i)\asymp f_\mu(r_n)a_n^{d-1}h_n
\quad\text{and}\quad 
\mu(\Delta_i^j)\asymp f_\mu(r_n)a_n^{d-1}h_n ,
\quad j=0,\ldots,d.
\]
\end{lemma}

\begin{proof}
By Lemma~\ref{lem:simplex-shell},
\[
\Delta_i\subseteq R_{t_n+1}(\mu)\setminus {\rm int}(R_{t_n}(\mu)).
\]
Therefore, for every \(x\in\Delta_i\),
\begin{equation}\label{eq:simplices-f-asymptotics}
e^{-1}f_\mu(r_n)\le f_\mu(x)\le f_\mu(r_n),
\end{equation}
so
\[
f_\mu(x)\asymp f_\mu(r_n).
\]
Since by \eqref{eq:simplices-volume-asymptotics}
\[
\vol_d(\Delta_i)\asymp a_n^{d-1}h_n,
\]
we get
\[
\mu(\Delta_i)
=
\int_{\Delta_i}f_\mu(x)\,dx
\asymp
f_\mu(r_n)a_n^{d-1}h_n
.
\]
The same proof works for \(\Delta_i^j\), since \(\Delta_i^j\subseteq \Delta_i\) and its volume is a fixed positive multiple of \(\vol_d(\Delta_i)\).
\end{proof}

We now define the relevant half-spaces. Let
\[
H_i^+:=\{x:\langle x,u_i\rangle\ge r_n\}.
\]
Observe that $\Delta_i \subseteq H_i^+$. Let also, for $1 \leq j \leq d$, $H_i^j$ be the half-space containing $\Delta_i^k$ for all $k=1,\ldots,d, \ k \neq j$, not containing $\Delta_i^0$ and $\Delta_i^j$, and whose bounding hyperplane touches $\Delta_i^k, \ k=0,\ldots,d, \ k \neq j$. 

Notice that \begin{equation}\label{eq:main-half-space-barycentric}H_i^+=\{x: \lambda_0(x) \geq 0\}.
\end{equation}We will express the other half-spaces in this way, as well.

\begin{lemma}\label{lem:half-spaces-barycentric}
    For every $1 \leq j \leq d$ we have that  \[H_i^j=\{x: s\lambda_0(x)+\lambda_j(x) \leq b_2\}, \quad s=\frac{b_2}{1-b_2}.\]
\end{lemma}

\begin{proof}
Fix $1 \leq j \leq d$. If $x \in 
\Delta_i^k$ for some $0 \leq k \leq d$, we have that \[\lambda(x)=(1-b_2)e_k+b_2\tau, \quad \tau \in \Lambda_d,\] where $\{e_0,\ldots,e_d
\}$ is the usual orthonormal basis of $\mathbb{R}^{d+1}$. If $k \in \{1,\ldots,d\} \setminus\{j\}$, then  \[s\lambda_0(x)+\lambda_j(x)=b_2(s\tau_0+\tau_j)\leq b_2,\] with equality for $\tau=e_j.$ If $x \in \Delta_i^0$, then \[s\lambda_0(x)+\lambda_j(x)=s(1-b_2)+b_2(\tau_0+\tau_j)=b_2(1+\tau_0+\tau_j) \geq b_2,\] with equality for $\tau=e_k, \, k \in\{1,\ldots,d\} \setminus \{j\}$. Now, if $x \in \Delta_i^j$, we have that \[s\lambda_0(x)+\lambda_j(x)=sb_2\tau_0+(1-b_2)+b_2\tau_j \geq 1-b_2>b_2.\]
\end{proof}

We need to upper bound the measure of all these half-spaces. We begin with the easy case, namely $\mu(H_i^+)$:

\begin{lemma}\label{lem:H-plus-upper-bound}
   There exists a constant $C(d)>0$ such that for every $i=1,\ldots,m$ \[\mu(H_i^+) \leq \frac{C(d)}{n}\] 
\end{lemma}

\begin{proof}
    Since $u_i \in S^{d-1}$, by the rotational invariance we have that \[\mu(H_i^+)=\mu(\{x: \langle x, e_1 \rangle \geq r_n\})=q_\mu(r_n)\] and the claim follows from \eqref{eq:q-one-over-n}.
\end{proof}

We proceed with $\mu(H_i^j)$. 

\begin{lemma}\label{lem:distance-from-origin}
    If $b_2>0$ is sufficiently small, then for every $j=1,\ldots,d$, we have that \[\dist(0,H_i^j) \geq r_n-L(d,b_2,c)h_n,\] where $L(d,b_2,c)>0$.
\end{lemma}

\begin{proof}
    Set \[\psi_j(x)=s\lambda_0(x)+\lambda_j(x), \quad x \in \mathbb{R}^d\] which is an affine function. Therefore, \[\psi_j(x)=\langle\nabla\psi_j,x \rangle+t\] for some $t \in \mathbb{R}$, where \[t=\psi_j(0)=s\lambda_0(0)+\lambda_j(0).\] With this notation, \begin{equation}\label{eq:distance-formula}\dist(0, H_i^j)=\frac{|b_2-\psi_j(0)|}{\|\nabla \psi_j\|_2}.
    \end{equation}
    Let $x \in \mathbb{R}^d$ and write \[x=\sum_{k=0}^d \lambda_k(x) y_i^k.\] If we take inner product with $u_i$, we get \[\lambda_0(x)=\frac{1}{h_n}(\langle x, u_i \rangle-r_n).
    \] In particular, \begin{equation}\label{eq:distance-first-coordinate}
    \lambda_0(0)=-\frac{r_n}{h_n}, \quad \nabla\lambda_0=\frac{1}{h_n}u_i.
    \end{equation} Taking again inner product with $w_j$ and using Proposition~\ref{prop:main-properties} ${\rm (iv)}$, we obtain \[\lambda_j(x)=\frac{1-\lambda_0(x)}{d}+\frac{d-1}{dc^2a_n^2}\langle x, w_j\rangle.\] More precisely, \begin{equation}\label{eq:distance-second-coordinate}
    \lambda_j(0)=\frac{r_n+h_n}{dh_n},\quad \nabla\lambda_j=-\frac{1}{dh_n}u_i+\frac{d-1}{dc^2a_n^2}w_j.
    \end{equation} Therefore, if $0<b_2<1/(d+1)$, we can easily check that $ds<1$ and $b_2d<1$. Hence, \begin{equation}\label{eq:sign-of-absolute}
        \psi_j(0)-b_2=\frac{(1-ds)r_n+(1-b_2d)h_n}{dh_n}>0,
    \end{equation} which implies that  $0\notin H_i^j$. Inserting \eqref{eq:distance-first-coordinate}, \eqref{eq:distance-second-coordinate} and \eqref{eq:sign-of-absolute} into \eqref{eq:distance-formula} yields \begin{align*}\dist(0, H_i^j)&=\frac{\frac{(1-ds)r_n+(1-b_2d)h_n}{dh_n}}{\sqrt{\left(s-\frac{1}{d}\right)^2 \frac{1}{h_n^2}+\frac{(d-1)^2}{d^2c^2a_n^2}}} \geq \\
    &\geq \frac{r_n}{\sqrt{1+\frac{(d-1)^2h_n^2}{(1-sd)^2c^2a_n^2}}}\geq\\
    & = r_n \left(1+\frac{(d-1)^2}{(1-sd)^2c^2}\frac{h_n}{r_n}\right)^{-1/2},
    \end{align*} where we used that $a_n^2=h_nr_n$. Applying the inequality \[(1+x)^{-1/2} \geq 1-\frac{x}{2}, \quad x \geq 0,\] concludes the proof.
\end{proof}

We are now able to upper bound the measure of $H_i^j$.

\begin{lemma}\label{lem:Hij-measure-upper-bound}
    There exists a constant \(C(d)>0\) and some $n_0:=n_0(\mu,d) \in \mathbb{N}$  such that for every $n \geq n_0$, every $i=1,\ldots,m$ and every $1 \leq j \leq d$
\[
\mu(H_i^j) \le \frac{C(d)}{n}.
\]
\end{lemma}
\begin{proof}
By the rotational invariance of $\mu$ we have that \[\mu(H_i^j)=q_\mu(\mathrm{dist}(0,H_i^j)) \leq q_\mu(r_n-L(d,b_2,c)h_n),\] where we used Lemma~\ref{lem:distance-from-origin} and the monotonicity of $q$. Now Lemma~\ref{lem:inward-q} completes the proof, as $L$ is independent of $i,j$ and $n$ and $b_2$ depends only on $d$.
\end{proof}

Assume now that $z_i^j$ is an arbitrary point of $\Delta_i^j$ for every $j=0,\ldots,d$.

\begin{lemma}\label{lem:z-affine-independence}
    The vectors $z_i^0,\ldots,z_i^d$ are affinely independent.
\end{lemma}

\begin{proof}
    For every $j=0,\ldots,d$, write \[z_i^j=(1-b_2)y_i^j+b_2x_i^j,\] where \[x_i^j=\sum_{k=0}^d \lambda_k(x_i^j)y_i^k, \quad j=0,\ldots,d.\] Then, for we have that \[z_i^j=\sum_{k=0}^d \lambda_k(z_i^j)y_i^k, \] where \[\lambda_j(z_i^j)=1-b_2+b_2\lambda_j(x_i^j) \geq 1-b_2\] and \[\lambda_k(z_i^j)=b_2\lambda_k(x_i^j), \quad k\in \{0,\ldots,d\} \setminus\{j\}.\] Consequently \[0 \leq \sum_{k \neq j}\lambda_k(z_i^j)=1-\lambda_j(z_i^j)\leq b_2.\] So if we consider the $(d+1) \times (d+1)$ matrix \[A=(\lambda_k(z_i^j))_{k,j=0}^d,\] then  $A$ is strictly diagonally dominant, as $0<b_2<1/2$, and thus invertible. This fact, combined with the affine independence of $y_i^0,\ldots,y_i^d$, proves the lemma.
\end{proof}

We now set $F_i=\conv\{z_i^1,\ldots,z_i^d\}$ and define the shadow behind $F_i$, as seen from $z_i^0$, by \[S(z_i^0,F_i)=\{z_i^0+t(x-z_i^0): x \in F_i, \, t\geq 1\}.\]

It is easy to see that $S(z_i^0,F_i)$ is  convex, since  $F_i$ is convex. Set also \[\widetilde{H}_i=H_i^+ \cup \bigcup_{j=1}^dH_i^j.\]

\begin{lemma}\label{lem:shadow-lemma}
If $b_2$ is chosen sufficiently small, then \[\mathbb{R}^d \setminus \widetilde{H}_i \subseteq S(z_i^0,F_i).\] 
\end{lemma}

\begin{proof}
    For $x \in \mathbb{R}^{d}$, we set \[\beta(x)=(\beta_0(x),\ldots,\beta_d(x)) \in \mathbb{R}^{d+1}\] to be the vector of barycentric coordinates of $x$, with respect to $z_i^0,\ldots,z_i^d$. This is unique due to Lemma~\ref{lem:z-affine-independence}. So, with this notation, $x \in S(z_i^0,F_i)$ is equivalent to \begin{equation}\label{eq:shadow-barycentric}\beta_0(x) \leq 0, \quad \beta_j(x) \geq 0, \quad j=1,\ldots,d.\end{equation} Furthermore, \begin{equation}\label{eq:barycentric-mixture}
    \lambda(x)=\sum_{j=0}^{d}\beta_j(x)\lambda(z_i^j),
    \end{equation}
    as the barycentric coordinates are affine functions. We remind that for every $k=0,\ldots, d$\[\lambda(z_i^k)=(1-b_2)e_k+b_2\tau_k, \quad \tau_k \in \Lambda_d,\] as $z_i^k \in \Delta_i^k$. Let now  \[x \in\mathbb{R}^d \setminus \widetilde{H}_i\] and assume that for some $1 \leq j \leq d$ we have that that \begin{equation}\label{eq:beta-negative-minimum}
    \beta_j(x)= \min_{1 \leq k \leq d} \beta_k(x)<0.
    \end{equation}
    Set also \[\psi_j(x)=s\lambda_0(x)+\lambda_j(x)\] and \[\delta_k=s\lambda_0(z_i^k)+\lambda_j(z_i^k)-b_2, \quad k=0,\ldots,d.\] By the calculations in Lemma~\ref{lem:half-spaces-barycentric}, we can see that
    \begin{equation}\label{eq:delta-inequalities}
    \begin{aligned}        
-b_2 &\leq \delta_k \leq 0, \quad k \in \{1,\ldots,d\}\setminus\{j\},\\
 0&\leq \delta_0 \leq b_2,\\
 1-2b_2& \leq \delta_j \leq 1-b_2.
\end{aligned}
\end{equation}
Notice that since $\sum_{k=0}^{d} \beta_k(x)=1$, we have that \[\psi_j(x)-b_2=\sum_{k=0}^d \beta_k(x)\delta_k .\] Using \eqref{eq:beta-negative-minimum} and \eqref{eq:delta-inequalities} we obtain \begin{equation}\label{eq:phi-upper-bound}
    \psi_j(x)-b_2 \leq -(d-1)b_2\beta_j(x)+(1-2b_2)\beta_j(x)+\beta_0(x)\delta_0.
\end{equation}

Next, we control the quantity $\beta_0(x)\delta_0$. If $\beta_0(x)<0$, then by \eqref{eq:delta-inequalities} \[\beta_0(x)\delta_0 \leq 0.\] Suppose now that $\beta_0(x)>0$. By \eqref{eq:barycentric-mixture} we get that \[\lambda_0(x)=\beta_0(x)\lambda_0(z_i^0)+ \sum_{k=1}^d \beta_k(x)\lambda_0(z_i^k).\] Now, for every $k=1,\ldots,d$ we have that \[0 \leq\lambda_0(z_i^k) \leq b_2\] as $z_i^k \in \Delta_i^k\subseteq \Delta_i$ . Moreover, because $z_i^0 \in \Delta_i^0$,  \[\lambda_0(z_i^0) \geq 1-b_2.\] Therefore,
\[\lambda_0(x) \geq (1-b_2)\beta_0(x)+db_2\beta_j(x),\] which yields \[\beta_0(x) \leq-\frac{db_2}{1-b_2}\beta_j(x),\] since $x  \notin H_i^+,$ which implies $\lambda_0(x)<0$ by \eqref{eq:main-half-space-barycentric}. Hence, by \eqref{eq:delta-inequalities} \[\beta_0(x)\delta_0 \leq \frac{-db_2^2}{1-b_2}\beta_j(x).\] Putting this estimate into \eqref{eq:phi-upper-bound} gives \[\psi_j(x)-b_2\leq -\beta_j(x)\left((d-1)b_2+-1+2b_2+\frac{db_2^2}{1-b_2}\right)<0,\] if $b_2$ is chosen sufficiently small. This is a contradiction by Lemma~\ref{lem:half-spaces-barycentric}, as $x \notin H_i^j$. Thus, \[\beta_j(x) \geq 0, \quad j=1,\ldots,d.\] Finally, if $\beta_0(x)>0$, then \[x \in \conv\{z_i^0,\ldots,z_i^d\} \subseteq \Delta_i \subseteq H_i^+,\] which is again a contradiction.
\end{proof}

\begin{lemma}\label{lem:geometric-small-c}
For every $i=1,\ldots,m$ define
\[
C_i:=z_i^0+\pos
\{z_i^1-z_i^0,\ldots,z_i^d-z_i^0\}.
\]
If we choose \(b_1>0\) sufficiently large and  \(b_2>0\) sufficiently small, then for every \(1 \leq i \neq \ell \leq m\) we have that
\[
\Delta_\ell\subseteq C_i.
\]
\end{lemma}

\begin{proof}
As $C_i$ is convex, it suffices to show that every vertex $x \in \Delta_\ell$ belongs to $C_i$. Observe that if $F_i=\conv\{z_i^1,\ldots,z_i^d\}$, it is enough to show that \[[z_i^0,x] \cap F_i \neq \emptyset.\] Since $r_n=\langle y_i,u_i\rangle$, the calculations of Lemma~\ref{lem:distance-from-origin} give \begin{equation}\label{eq:barycentric-coordinates-formula}
\lambda_0(x)=\frac{1}{h_n}\langle x-y_i,u_i\rangle, \quad \lambda_j(x)=\frac{1-\lambda_0(x)}{d}+\frac{d-1}{dc^2a_n^2}\langle x,w_j\rangle,\end{equation} where $w_j \in u_i^\perp$ and $\|w_j\|_2=ca_n$ for every $j=1,\ldots,d.$ Set now $s=\|y_\ell-y_i\|_2\geq 2b_1a_n$. We can write \[x=y_\ell^k=y_\ell+\zeta,\] where if $k=0$, \[\zeta=\frac{h_n}{r_n}\,y_\ell, \quad \|\zeta\|_2=h_n\] and if $k=1,\ldots, d,$ \[\zeta \in y_\ell^\perp, \quad \|\zeta\|_2=ca_n. \] Since \[\langle y_\ell-y_i,u_i\rangle=r_n(\langle u_\ell,u_i\rangle-1),\] we get that \[s^2=r_n^2\,\|u_\ell-u_i\|_2^2=2r_n^2(1-\langle u_\ell,u_i \rangle)=-2r_n \langle y_\ell-y_i,u_i\rangle.\] Hence, \eqref{eq:barycentric-coordinates-formula} yields
\[\lambda_0(x)=-\frac{s^2}{2r_nh_n}+\frac{1}{h_n}\langle \zeta,u_i\rangle.\]
If $k=0$ we have that \[|\langle \zeta,u_i \rangle| \leq h_n,\] whereas if $k=1,\ldots,d$ \[|\langle \zeta, u_i\rangle|=|\langle \zeta, u_i-u_\ell\rangle| \leq \frac{ca_n}{r_n}s.\] To sum up,\begin{equation}\label{eq:zeta-inner-upper-bound}
    \langle \zeta,u_i \rangle \leq |\langle \zeta,u_i \rangle| \leq h_n+\frac{ca_n}{r_n}s.
\end{equation}Therefore, since $a_n^2=r_nh_n$, \begin{equation}\label{eq:0-coordinate-upper-bound}    
\lambda_0(x) \leq-\frac{s^2}{2a_n^2}+\frac{cs}{a_n}+1 \leq-\frac{s^2}{4a_n^2},\end{equation}if we choose $b_1>0$ large enough so that \[c\,\frac{s}{a_n}+1 \leq \frac{1}{4}\left(\frac{s}{a_n}\right)^2.\] It is clear that we can  do so, independently of $n$, as \[\frac{s}{a_n} \geq 2b_1.\] In addition, using again \eqref{eq:zeta-inner-upper-bound}, we see that \begin{equation}\label{eq:absolute-0-coordinate}\begin{aligned}
    |\lambda_0(x)| &\leq \frac{s^2}{2a_n^2}+c\,\frac{s}{a_n}+1 \leq \frac{s^2}{a_n^2},
    \end{aligned}
\end{equation} if $b_1>0$ is large enough.

For $j=1,\ldots,d$ using Cauchy-Schwarz, and considering both cases for $\|\zeta\|_2$, we get \[|\langle x,w_j\rangle|=|\langle x-y_i,w_j \rangle| \leq|\langle y_\ell-y_i,w_j \rangle| +|\langle \zeta,w_j \rangle| \leq ca_n(s+ca_n+h_n).\] Combining with the first inequality of \eqref{eq:0-coordinate-upper-bound}, \eqref{eq:barycentric-coordinates-formula} yields \begin{equation}\label{eq:j-barycentric-lower-bound}
    \begin{aligned}
        \lambda_j(x) & \geq \frac{1}{d}\left[ \frac{s^2}{2a_n^2}-c\,\frac{s}{a_n}-\frac{d-1}{ca_n}\left(s+ca_n+h_n\right)\right]\\
        &=\frac{1}{d}\left[\frac{1}{2}\left(\frac{s}{a_n}\right)^2-\frac{d-1+c^2}{c}\frac{s}{a_n}-(d-1)\left(1+\frac{h_n}{ca_n}\right)\right] \\
        &\geq \frac{s^2}{4da_n^2}, 
    \end{aligned} 
\end{equation} if we choose $b_1>0$ sufficiently large, so that \[\frac{1}{2}\left(\frac{s}{a_n}\right)^2-\frac{d-1+c^2}{c}\frac{s}{a_n}-(d-1)\left(1+\frac{h_n}{ca_n}\right)\geq \frac{1}{4}\left(\frac{s}{a_n}\right)^2,\] which we can do, independently of $n$, since $s/a_n \geq 2b_1$ and $(h_n/a_n)_{n\in \mathbb{N}}$ is a bounded sequence. Hence, by \eqref{eq:absolute-0-coordinate} and \eqref{eq:j-barycentric-lower-bound}  \[\|\lambda(x)\|_1=|\lambda_0(x)|+\sum_{j=1}^{d}\lambda_j(x) =|\lambda_0(x)|+1-\lambda_0(x) \leq 3 \frac{s^2}{a_n^2},\] if $b_1>0$ is large enough. Thus, by \eqref{eq:j-barycentric-lower-bound} \begin{equation}\label{eq:1-norm-comparing}
    \lambda_j(x) \geq \frac{\|\lambda(x)\|_1}{12d}, \quad j=1,\ldots,d.
\end{equation} Consider again the $(d+1)\times(d+1)$ matrix \[A=(\lambda_k(z_i^j))_{k,j=0}^d\] from Lemma~\ref{lem:z-affine-independence}, whose columns are the barycentric vectors of $z_i^j$ with respect to $\Delta_i$. Notice that we can write \[A=I_{d+1}+E,\] where each column of $E$ has $\ell_1$-norm at most $2b_2$. Therefore, \[\|E:\ell_1^{d+1}\to \ell_1^{d+1}\| \leq 2b_2\] and since $0<b_2<1/2$, the Neumann series gives \begin{equation}\label{eq:neumann-result}
    \|(A^{-1}-I_{d+1}): \ell_1^{d+1} \to \ell_1^{d+1}\| \leq \frac{\|E:\ell_1^{d+1}\to \ell_1^{d+1}\|}{1-\|E:\ell_1^{d+1}\to \ell_1^{d+1}\|} \leq \frac{2b_2}{1-2b_2}.
\end{equation} Now, if $\beta(x)$ is the barycentric vector of $x$, with respect to  $z_i^0,\ldots,z_i^d$, we have that \[\lambda(x)=A\beta(x).\] Consequently, by \eqref{eq:neumann-result} \[\|\beta(x)-\lambda(x)\|_1 \leq \frac{2b_2}{1-2b_2}\|\lambda(x)\|_1.\] Thus, from \eqref{eq:1-norm-comparing}, if we choose $b_2>0$ sufficiently small,  then for every $j=1,\ldots, d$, \begin{align*}
    \beta_j(x)&\geq \lambda_j(x) -|\beta_j(x)-\lambda_j(x)| \\
    &\geq \lambda_j(x)-\|\beta(x)-\lambda(x)\|_1\\
    &\geq \left(\frac{1}{12d}-\frac{2b_2}{1-2b_2}\right)\| \lambda(x)\|_1>0.  
\end{align*}
To conclude, \[x=\beta_0(x)z_i^0+\sum_{j=1}^d \beta_j(x)z_i^j=z_i^0+\sum_{j=1}^d\beta_j(x)(z_i^j-z_i^0) \in C_i,\] as $\sum_{j=0}^d\beta_j(x)=1$.
\end{proof}

\begin{corollary}\label{cor:independence-under-condition}
    If $b_1>0$ is large enough, then for every $1 \leq i \neq \ell \leq m$, we have that \[\mu\bigl(\Delta_\ell \cap \widetilde{H}_i\bigr)=0\]
\end{corollary}

\begin{proof}
    Let $1 \leq j \leq d$. Since the boundary of $H_i^j$ touches $\Delta_i^k, \, k \in \{0,\ldots,d\} \setminus\{j\}$, we can choose points \[z_i^k \in \Delta_i^k \cap \partial H_i^j \quad k \in \{0,\ldots,d\}\setminus\{j\}.\] Choose also some $z_i^j \in \Delta_i^j$. With this selection, by the definition of $H_i^j$, we can see that \[z_i^0+\pos\{z_i^k-z_i^0:k=1,\ldots d\} \subseteq {\rm int}(H_i^j)^{c}.\] Therefore, Lemma~\ref{lem:geometric-small-c} yields \[\Delta_\ell \cap H_i^j \subseteq \partial H_i^j,\] if $b_1>0$ is chosen large enough. We will now show that \[\Delta_\ell \cap H_i^+=\emptyset,\] where we remind that \[H_i^+=\{x: \langle x,y_i \rangle \geq r_n^2\}.\] It suffices to show that if $x \in \Delta_\ell$ is some of its vertices, then $x \notin H_i^+$. Recall first that \[\|y_i-y_\ell\|_2 \geq 2b_1a_n,\] which implies that \[\langle y_i,y_\ell \rangle \leq r_n^2-2b_1^2r_nh_n. \] If $x=y_{\ell}^0=(1+h_n/r_n)y_\ell$, then \[\langle x,y_i\rangle \leq r_n^2+h_nr_n(1-2b_1^2)-2b_1^2h_n^2 <r_n^2,\] if $b_1>1/\sqrt{2}$. Assume now that $x=y_\ell^k=y_\ell+w_k,$ for some $1 \leq k \leq d$, where \[w_k \perp y_\ell, \quad \|w_k\|_2=ca_n.\] By Cauchy-Schwarz,
    \begin{align*}
        \langle y_i,x \rangle &=\langle y_i,y_\ell\rangle +\langle y_i,w_k\rangle=\langle y_i,y_\ell\rangle +\langle y_i -y_\ell,w_k\rangle\\
        & \leq r_n^2- \frac{\|y_i-y_\ell\|_2^2}{2}+ca_n \|y_i-y_\ell\|_2 \leq r_n^2+(c-b_1)a_n\|y_i-y_\ell\|_2\\
        &<r_n^2,
    \end{align*} if $b_1>c$. Thus, \[ \Delta_\ell \cap \widetilde{H}_i \subseteq \bigcup_{j=1}^{d} \partial H_i^j\] and the result follows, because $\mu$ is absolutely continuous with respect to Lebesgue measure.
\end{proof}

\section{Lower bound for the intrinsic volumes}

In this section we prove Theorem~\ref{thm:main-general-radial}. To this end, we establish Proposition~\ref{prop:variance-lower-bound}, which  is a local variance lower bound for the intrinsic volume contribution generated by a single configuration, with all remaining points fixed. This provides the basic fluctuation estimate that is later summed over the many well-separated good configurations.

However, in order to do so, we first need to collect a few preliminary facts. For the rest of the section, we fix some $\ell \in \{1,\ldots,d\}$. The following lemma is standard; nevertheless, we include a proof for completeness.

\begin{lemma}\label{lem:angle-and-projection}
    Let $u \in S^{d-1}$ and $L \in G_{d,\ell}$ . If $\angle(u,L)=\theta$, then \[\cos \theta =\|P_Lu\|_2, \quad \sin\theta= \|P_{L^\perp}u\|_2.\]
\end{lemma}

\begin{proof}
If $u \in L^\perp$, then it is obvious that \[\cos\theta=0=\|P_Lu\|_2.\] Assume now that $u \notin L^\perp$.
    As $u \in S^{d-1}$ we have that \[\theta=\min_{v \in S_L} \, \arccos|\langle u,v \rangle|.\] If we decompose \[u=P_Lu+P_{L^\perp}u,\] then for every $v \in S_L$, Cauchy-Schwarz yields \[|\langle u,v\rangle|=|\langle P_Lu,v\rangle| \leq \|P_Lu\|_2\] and we have equality when \[v=\frac{P_Lu}{\|P_Lu\|_2} \in S_L.\] Therefore, as $x \mapsto\arccos x$ is decreasing: \[\theta =\arccos \bigl(\,\max_{v \in S_L} |\langle u,v\rangle|\bigr)=\arccos\bigl(\|P_Lu\|_2\bigr).\] To conclude, notice that since $u \in S^{d-1}$ and $\theta \in [0,\pi/2]$: \[\sin\theta=\sqrt{1-\cos^2\theta}=\sqrt{\|u\|_2^2-\|P_Lu\|_2^2}=\|P_{L^\perp}u\|_2.\]
\end{proof}

The next geometric fact has appeared in \cite{Bar-Fod-Vig} and \cite{Bar-Th}. We give a more probabilistic proof.

\begin{lemma}\label{lem:small-angles-measure}
    Let $u \in S^{d-1}$ . There exists a constant $c_1>0$ such that \[\nu_{d,\ell}(\{L \in G_{d,\ell}: \angle (u,L) \leq t\}) \gtrsim t^{d-\ell}, \] for every $0<t<c_1$.
\end{lemma}

\begin{proof}
    If $\ell=d$, then $G_{d,\ell}=\{\mathbb{R}^d\}$ and $\angle(u,\mathbb{R}^d)=0$. Hence, \[\nu_{d,\ell}(\{L \in G_{d,\ell}: \angle (u,L) \leq t\})=1=t^{d-\ell},
    \] for every $t>0$. Assume now that $1 \leq \ell \leq d-1$. It is obvious  that if $0<t<\pi/2$, then \begin{equation}\label{eq:integral-sine-representation}
        \nu_{d,\ell}(\{L \in G_{d,\ell}: \angle (u,L) \leq t\})=\nu_{d,\ell}(\{L \in G_{d,\ell}: \sin^2(\angle(u,L)) \leq \sin^2t\})
    \end{equation} Let also $\{e_1,\ldots,e_d\}$ be the usual orthonormal basis of $\mathbb{R}^d$ and set $E_\ell={\rm span}\{e_1,\ldots,e_\ell\}$. If $Q \sim \nu$, where $\nu$ is the rotationally invariant Haar probability measure on $O(d)$, then the same holds true for $Q^T$. Therefore \[Q^Tu \sim {\rm Unif}(S^{d-1}), \quad L=Q\,E_\ell\sim\nu_{d,\ell}.\] Write now \[Q^Tu=(U_1,\ldots,U_d)\] and observe that \[\|P_{L^\perp}u\|_2^2 =\sum_{j=\ell+1}^d \langle u,Qe_j \rangle^2=\sum_{j=\ell+1}^{d} U_j^2.\] Now if $g_1,\ldots,g_d$ are independent $N(0,1)$ random variables, we have that \[Q^Tu\overset{d}{=} \frac{(g_1,\ldots,g_d)}{\|g\|_2}, \] which implies that \begin{equation}\label{eq:beta-representation}
        \|P_{L^\perp}u\|_2^2 \overset{d}{=}\frac{g_{\ell+1}^2+\ldots+g_d^2}{g_1^2+\ldots+g_d^2} \sim {\rm Beta}(\alpha,\beta), \quad \alpha=\frac{d-\ell}{2}, \, \beta=\frac{\ell}{2}.
    \end{equation} Combining \eqref{eq:integral-sine-representation} and \eqref{eq:beta-representation} with Lemma~\ref{lem:angle-and-projection}, we obtain 
    \begin{equation}\label{eq:sine-integral}
       \nu_{d,\ell}(\{L \in G_{d,\ell}: \angle (u,L) \leq t\})=B(\alpha,\beta)^{-1} \int_{0}^{\sin^2 t} s^{\alpha-1}(1-s)^{\beta-1} \, ds.
    \end{equation}
    We choose $c_1=\pi/4$. So, if $0<t<c_1$ and $0 \leq s \leq \sin^2 t$, we have that  \[1-s \geq \frac{1}{2}, \quad \sin t \geq \frac{t}{2}.\] Consequently, \begin{equation}\label{eq:beta-integral-inequalities}
        \int_{0}^{\sin^2 t} s^{\alpha-1}(1-s)^{\beta-1} \, ds \geq \frac{2^{-\beta+1}}{\alpha} \sin^{2\alpha}(t)\geq \frac{2^{-2\alpha-\beta+1}}{\alpha} t^{d-\ell}.    
    \end{equation} Consider now that $\alpha,\beta$ are functions of $d, \ell$. Hence, \eqref{eq:sine-integral} and \eqref{eq:beta-integral-inequalities} yield the lemma.
\end{proof}

We will also need the following lemma, which is presumably well known. However, since we were unable to locate a proof in the literature, we record one below. A dual version of this result, appears in \cite[Theorem~2]{MTTT}.

\begin{lemma}\label{lem:projection-volume-lipschitz}
    Let $K \subseteq \mathbb{R}^d$ be a convex body with $D=\diam(K)$. If  $L,M \in G_{d,\ell}$, then \[|\vol_\ell (P_LK)-\vol_\ell (P_MK)| \leq C_{\ell} D^{\ell} d_{G_{d,\ell}}(L,M),\] where $C_\ell>0$ is a constant depending only on $\ell$.
\end{lemma}

\begin{proof}
    By translation, we can assume that $K \subseteq DB_2^d$. Set \[\delta= d_{G_{d,\ell}}(L,M).\] If $\delta=1$, then we have that \[P_LK \subseteq DB_{L}, \quad P_MK \subseteq DB_M.\] Therefore, \[|\vol_\ell (P_LK)-\vol_\ell (P_MK)| \leq \max\{\vol_\ell (P_LK),\vol_\ell (P_MK)\} \leq \omega_{\ell}D^{\ell}=\omega_\ell D^\ell\delta.\] So we can assume that $0 < \delta<1$. Set 
    \[T:L \to M, \quad T=P_M|_L.\]
    If $x \in L$, then $P_Lx=x$. So, \[\|Tx-x\|_2=\|P_Mx-P_Lx\|_2\leq \delta \|x\|_2.\] Hence, \[\|Tx\|_2\geq \|x\|_2-\|Tx-x\|_2 \geq (1-\delta)\|x\|_2,\] which means that \[(1-\delta)\|x\|_2\leq \|Tx\|_2\leq \|x\|_2.\] Then $T$ is an isomorphism, as $L,M$ have the same dimension. Moreover, if $s_1,\ldots,s_\ell$ are the singular values of $T$, then \[1-\delta \leq s_i \leq 1, \quad i=1,\ldots,\ell.\] So we can deduce that \begin{equation}\label{eq:determinant-around-1}
    0 \leq 1-|{\rm det}T| \leq 1-(1-\delta)^\ell \leq \ell \delta.
    \end{equation} Observe now that \[P_M-TP_L=P_M(I_n-P_L)=(P_M-P_L)(I_n-P_L)=(P_M-P_L)P_{L^\perp},\] as $P_L=P_L^2$. Consequently, \[\|P_M-TP_L\|\leq \|P_M-P_L\| \|P_{L^\perp}\|=\delta.\] Now, if $x \in K$, we have that \[\|(P_M-TP_L)x\|_2 \leq \delta D,\] or equivalently \[d_{H}(P_MK,TP_LK) \leq \delta D.\] In addition, \[P_MK, TP_LK \subseteq DB_M.\] A well known argument in convex geometry implies that \begin{equation}\label{eq:first-volume-bound}
    |\vol_\ell (P_MK)-\vol_\ell(TP_LK)|  \leq C_\ell\, D^{\ell-1} d_{H}(P_MK,TP_LK) \leq c_\ell\, D^\ell \delta,
    \end{equation} where $c_\ell>0$ is a constant depending only on $\ell$. However, \[\vol_\ell(TP_LK)=|{\rm det}T|\,\vol_\ell(P_LK).\] So by \eqref{eq:determinant-around-1} \[|\vol_{\ell}(P_LK)-\vol_\ell(TP_LK)|=(1-|{\rm det}T|)\, \vol_{\ell}(P_LK) \leq\ell\,\omega_\ell D^\ell \delta.\] Combining with \eqref{eq:first-volume-bound}, we get the proof.
\end{proof}

Recall now that for some fixed $i=1,\ldots,m$, the vectors $z_i^j$ are arbitrary points of $\Delta_i^j, \ j=0,\ldots,d$. We have also set \[F_i =\conv\{z_i^1,\ldots,z_i^d\}.\] 

For a convex body $K \subseteq \mathbb{R}^d$ we denote by $\mu_K$ the normalized restriction of $\mu$ to $K$, i.e. \[\mu_K(B)=\frac{\mu(B \cap K)}{\mu(K)}, \quad B \in \mathcal{B}(\mathbb{R}^d).\] Observe that $\mu_K$ is always well defined, as $\mu$ has full support and $K$ has non-empty interior.

\begin{proposition}\label{prop:variance-lower-bound}
    Fix  $1 \leq i \leq m$. Let $Z$ be a random vector on $\Delta_i^0$, with $Z \sim \mu_{\Delta_i^0}$. Then, \[\Var \bigl(V_\ell(\conv\{Z,F_i\})\bigr) \gtrsim a_n^{4\ell-2d-2}h_n^{2d-2\ell+2}\]
\end{proposition}

\begin{proof}

   The proof will be based on a series of claims. Set $W=\conv\{w_1,\ldots,w_d\}$.
    A quick calculation shows that we can write \begin{equation}\label{eq:delta0-alternative-writing}
    \Delta_i^0=\left\{y_i+tu_i+p: \, (1-b_2)h_n \leq t \leq h_n, \, p \in \left(1-\frac{t}{h_n}\right)W\right\}.
    \end{equation}This means that for some $p \in b_2W$, the vertical fiber of $\Delta_i^0$ above $p$ is \[J(p)=[(1-b_2)h_n,(1-p_W(p))\,h_n],\] where $p_W$ is the Minkowski functional of $W$. In particular, if $p \in \frac{b_2}{2}W$, then \begin{equation}\label{eq:fiber-length}
      \left[ (1-b_2)h_n, \left(1-\frac{b_2}{2}\right)h_n \right] \subseteq J(p) 
    \end{equation}
     In a similar way, for every $j=1,\ldots,d$ we can write \[z_i^j=y_i+s_ju_i+v_j,\] where \begin{equation}\label{eq:w-v-distances}
     0 \leq s_j \leq b_2h_n, \quad v_j \in u_i^\perp, \quad \|w_j-v_j\|_2 \leq 2b_2ca_n.
     \end{equation}

     \begin{claim}\label{claim:affine-independence}
         If $b_2$ is chosen sufficiently small, then the vectors $v_1,\ldots,v_d$ are affinely independent.
     \end{claim}

     \begin{proof}
         Define the linear map \[T: \mathbb{R}^{d-1} \to u_i^\perp, \quad Tx=\sum_{j=1}^{d-1}(w_j-w_d)x_j,\] which is invertible as $w_1,\ldots,w_d$ are affinely independent. Since $\|w_j\|_2=ca_n$ for every $j=1,\ldots d$, and they form a regular $(d-1)$- simplex, centered at $0$, one can see that \begin{equation}\label{eq:M-norms}
         \|T\|=\frac{ca_nd}{\sqrt{d-1}}, \quad \|T^{-1}\|=\frac{1}{ca_n}\sqrt{\frac{d}{d-1}}.
         \end{equation}
         Let also \[\Gamma: \mathbb{R}^{d-1} \to u_i^\perp, \quad \Gamma x=\sum_{j=1}^{d-1}(v_j-v_d)x_j.\] Observe that by \eqref{eq:w-v-distances} we get \[\|T-\Gamma\| \leq 4b_2ca_n \sqrt{d-1}\]  Then for every $x \in S^{d-2}$ we have 
         \begin{equation}\label{eq:N-norm-inverse}
         \|\Gamma x\|_2 \geq  \|Tx\|_2-\|(\Gamma-T)x\|_2 \geq ca_n\sqrt{\frac{d-1}{d}}-4b_2ca_n\sqrt{d-1} \geq ca_n \sqrt{\frac{d-1}{2d}},
         \end{equation} if $b_2$ is chosen sufficiently small. Therefore $\Gamma$ is invertible and \begin{equation}\label{eq:N-inverse-norm}
             \|\Gamma^{-1}\| \leq \frac{1}{ca_n} \sqrt{\frac{2d}{d-1}}.
         \end{equation} The claim follows.
     \end{proof}

In fact, we can show something stronger, but we will need again the previous calculations, which makes the last proof useful.

     \begin{claim}\label{claim:ball-inside-v}
     If $b_2$ is chosen small enough, then \[\frac{ca_n}{2(d-1)} B_{u_i^\perp} \subseteq\conv\{v_1,\ldots,v_d\}\]     
     \end{claim}

     \begin{proof}
         By Proposition~\ref{prop:main-properties} ${\rm (vii)}$, we have that \[ \frac{ca_n}{d-1} B_{u_i^\perp} \subseteq W.\] Then, applying \eqref{eq:w-v-distances}, for every $\theta \in S_{u_i^\perp}$ and every $1 \leq j \leq d$ we have that 
         \[\langle v_j,\theta\rangle \geq \langle w_j,\theta\rangle-2b_2ca_n.\] Taking maximum over $1 \leq j \leq d$ yields \[h_{\conv\{v_1,\ldots,v_d\}}(\theta) \geq \frac{ca_n}{d-1}-2b_2ca_n \geq \frac{ca_n}{2(d-1)},\] if $b_2$ is chosen small enough.
     \end{proof}

     So by Claim~\ref{claim:affine-independence} and \eqref{eq:w-v-distances} we get that \[u_i^\perp ={\rm span}\{v_1-v_d,\ldots,v_{d-1}-v_d\}.\] In fact, since $v_d \in u_i^\perp$, we also have that \[u_i^\perp={\rm aff}\{v_1,\ldots,v_d\}\] Set now \[H_{F_i}={\rm aff}(F_i).\] Notice that $P_{u_i^\perp}(z_i^j)=v_j$ for every $j=1,\ldots,d$. Therefore, \[P_{u_i^\perp}:H_{F_i} \to u_i^\perp \] is an affine bijection. So, there exists unique affine function \[g: u_i^\perp \to \mathbb{R},\] such that \begin{equation}\label{eq:HF-writing}    
     H_{F_i}=\{y_i+p+g(p)u_i: p \in u_i^\perp\}.
     \end{equation}In particular, \[g(v_j)=s_j, \quad j=1,\ldots d.\] Since $g$ is affine, it can be written as \[g(p)=\langle \nabla g,p \rangle+b,\] where $\nabla g \in u_i^\perp$.

     \begin{claim}\label{claim:nabla-upper-bound}
         We have that \[\|\nabla g \|_2 \leq C(d)b_2 \frac{h_n}{ca_n},\] where $C(d)>0$.
     \end{claim}

     \begin{proof}
         For every $j=1,\ldots,d-1$ we have that \[\langle \nabla g, v_j-v_d\rangle=g(v_j)-g(v_d)=s_j-s_d.\] Hence, if \[\tilde{s}=(s_1-s_d,\ldots,s_{d-1}-s_d) \in \mathbb{R}^{d-1},\] the above equations can be written as \[\Gamma^\ast \nabla g=\tilde{s},\] where $\Gamma^\ast: u_i^\perp \to \mathbb{R}^{d-1}$ is the adjoint of $\Gamma$, in the proof of Claim~\ref{claim:affine-independence}. Using \eqref{eq:w-v-distances}, we obtain \[\|\tilde{s}\|_2 \leq b_2h_n\sqrt{d-1}.\] Hence by \eqref{eq:N-inverse-norm} \[\|\nabla g\|_2 =\|(\Gamma^\ast)^{-1}\tilde{s}\|_2 \leq \|(\Gamma^\ast)^{-1}\| \,\|\tilde{s}\|_2 \leq \frac{b_2h_n}{ca_n}\sqrt{2d}, \] which concludes the claim.
     \end{proof}

     \begin{claim}\label{claim:sign-of-difference}
     Let $z=y_i+p+tu_i \in \Delta_i^0$, as in \eqref{eq:delta0-alternative-writing}. If $b_2$ is sufficiently small, \[t-g(p)>0.\]       
     \end{claim}

     \begin{proof}
         By \eqref{eq:delta0-alternative-writing} we have that \[t \geq h_n(1-b_2), \quad \|p\|_2 \leq b_2ca_n.\] Recall that $g(p)=b+\langle \nabla g, p\rangle$ and $0 \in \conv\{v_1,\ldots,v_d\}$ by Claim~\ref{claim:ball-inside-v}. Consequently,  \[b=g(0) \leq \max_{1 \leq j \leq d} g(v_j) =\max_{1 \leq j \leq d} s_j \leq b_2h_n,\] where we used that $g(v_j)=s_j$ and \eqref{eq:w-v-distances}. Thus, by Claim~\ref{claim:nabla-upper-bound} \[t-g(p) \geq h_n(1-b_2)-b_2h_n-C(d)b_2^2h_n=(1-2b_2-C(d)b_2^2\,)\,h_n>0,\] if $b_2$ is chosen small enough.
     \end{proof}

     Set now $N_{F_i}=u_i-\nabla g.$ Let two arbitrary points of $H_{F_i}$ \[x_1=y_i+p_1+g(p_1)u_i, \quad x_2=y_i+p_2+g(p_2)u_i,\] where $p_1,p_2 \in u_i^\perp$. Then \[x_1-x_2=p_1-p_2+\langle \nabla g,p_1-p_2\rangle u_i.\] Therefore, as $\nabla g \in u_i^\perp$, \[\langle N_{F_i}, x_1-x_2\rangle=0,\] which means that $N_{F_i}$ is perpendicular to $H_{F_i}$. 

     \begin{claim}\label{claim:ball-in-F}
         There exists some $x \in F_i$ such that \[B_{H_{F_i}}\left(x,\frac{ca_n}{2(d-1)}\right) \subseteq F_i. \] 
     \end{claim}

     \begin{proof}
         We define \[\Phi: u_i^\perp \to H_{F_i}, \quad \Phi(p)=y_i+p+g(p)u_i,\] which is clearly surjective. As $g$ is affine and $\Phi(v_j)=z_i^j$ for every $1 \leq j \leq d$, we can deduce that  \[\Phi(\conv\{v_1,\ldots,v_d\})=F_i.\] Write now \[\Phi(p)=\Phi(0)+T_gp, \quad T_gp=p+\langle \nabla g,p\rangle u_i, \quad p\in u_i^\perp.\] Then for every $p \in u_i^\perp$, \[ \|T_gp\|_2^2=\|p\|_2^2+\langle \nabla g , p\rangle ^2 \geq \|p\|_2^2.\] We are now able to show that \[B_{H_{F_i}}\left(\Phi(0),\frac{ca_n}{2(d-1)}\right) \subseteq \conv\{v_1,\ldots,v_d\}.\] Indeed, as $\Phi$ is surjective, let $\Phi(p) \in B_{H_{F_i}}\bigl(\Phi(0),\frac{ca_n}{2(d-1)}\bigr)$ for some $p \in u_i^\perp$. Then \[\frac{ca_n}{2(d-1)} \geq\|\Phi(p)-\Phi(0)\|_2=\|T_gp\|_2 \geq \|p\|_2,\] which implies that $p \in \frac{ca_n}{2(d-1)} B_{u_i^\perp}$. By Claim~\ref{claim:ball-inside-v} we get that  $p \in \conv\{v_1,\ldots,v_d\}$ and thus \[\Phi(p) \in F_i.\]
     \end{proof}

     Fix now \[p \in \frac{b_2}{2}W.\]

     \begin{claim}\label{claim:p-line-intersection}
         If $b_2$ is sufficiently small, then the vertical line $y_i+p+\mathbb{R}u_i$ intersects $H_{F_i}$ at $\Phi(p) \in F_i$. Moreover, \[\langle \Phi(p), u_i \rangle < \min_{x \in \Delta_i^0} \langle x, u_i\rangle, \] which means that $\Phi(p)$ lie strictly below $\Delta_i^0$, for every $ p \in \frac{b_2}{2}W$.
     \end{claim}

     \begin{proof}
     Because of \eqref{eq:HF-writing}, the line intersects $H_{F_i}$ at \[\Phi(p)=y_i+p+g(p)u_i.\] We have to show that $\Phi(p) \in F_i$. It suffices to prove that $p \in \conv\{v_1,\ldots,v_d\}$. Indeed, if $b_2$ is sufficiently small, as $\|w_j\|_2=ca_n$ for every $1 \leq j \leq d$, we have that \[p \in \frac{b_2}{2}W  \subseteq \frac{b_2ca_n}{2}B_{u_i^\perp} \subseteq \frac{ca_n}{2(d-1)}B_{u_i^\perp} \subseteq \conv\{v_1,\ldots,v_d\}\] by Claim~\ref{claim:ball-inside-v}.

     For the second assertion, since $g$ is affine and $p \in \conv\{v_1,\ldots,v_d\}$, using \eqref{eq:w-v-distances}, we have \[\langle \Phi(p),u_i \rangle=r_n+ \langle \Phi(p)-y_i,u_i \rangle =r_n+g(p) \leq r_n+ \max_{1 \leq j \leq d} s_j \leq r_n+ b_2h_n,\] while on the other hand
     \[\min_{x \in \Delta_i^0} \langle x, u_i \rangle =(1-b_2)(r_n+h_n)+b_2\min_{0 \leq j \leq d} \langle y_i^j,u_i\rangle=r_n+(1-b_2)h_n.\] The claim follows, since $b_2<1/2$.        
     \end{proof}

     Set now \[I_1(p)= \left[ \left(1-\frac{5b_2}{8}\right)h_n, \left(1-\frac{b_2}{2}\right)h_n \right].\]
     and
     \[I_2(p)= \left[ \left(1-\frac{7b_2}{8}\right)h_n , \left(1-\frac{3b_2}{4}\right)h_n \right]\] 
       Observe by \eqref{eq:fiber-length} that we have  \[I_1(p),I_2(p) \subseteq J(p).\]

     Let also $t_1 \in I_1(p)$ and $t_2 \in I_2(p).$ Then \begin{equation}\label{eq:tj-difference}
      t_1-t_2 \geq \frac{b_2}{8}h_n.  
     \end{equation}
     In addition, set \[z_1=y_i+p+t_1u_i, \quad z_2=y_i+p+t_2u_i.\]

    We note that by the deinition of $J(p)$ \[z_1,z_2 \in \Delta_i^0.\]

     \begin{claim}\label{claim:convex-hull-inclusion}
         We have that \[\conv\{z_2,F_i\}\subseteq \conv\{z_1,F_i\}.\] In particular, if  $L \in G_{d,\ell}$, then \[\vol_{\ell}(P_L \conv\{z_1,F_i\}) \geq \vol_{\ell}(P_L \conv\{z_2,F_i\}).\]
     \end{claim}

     \begin{proof}
         Since $z_2 \in \Delta_i^0$, by Claim~\ref{claim:p-line-intersection} we can write \[\Phi(p)=z_2-su_i,\] for some $s>0$. Moreover, \[z_1=z_2+(t_1-t_2)u_i.\] Solving for $z_2$ we get \[z_2=\frac{s}{s+t_1-t_2}\,z_1+\frac{t_1-t_2}{s+t_1-t_2}\, \Phi(p),\] as a convex combination. Therefore, by Claim~\ref{claim:p-line-intersection} \[z_2 \in [z_1,\Phi(p)] \subseteq \conv\{z_1,F_i\},\] which concludes the proof.
     \end{proof}

     \begin{claim}\label{claim:subspaces-containing-NF}
        Let $L_0 \in G_{d, \ell}$ that contains $N_{F_i}$. Then \[\vol_{\ell}(P_{L_0} \conv\{z_1,F_i\}) - \vol_{\ell}(P_{L_0} \conv\{z_2,F_i\}) \gtrsim h_na_n^{\ell-1}.\]
     \end{claim}

     \begin{proof}
         Set \[M_0= L_0 \cap N_{F_i}^\perp \in G_{d,\ell-1}.\] We can decompose as \[L_0=M_0 \oplus {\rm span} \{N_{F_i}\}.\] Since $N_{F_i}$ is perpendicular to $H_{F_i}$, we can write \[H_{F_i}=\{y: \langle y, N_{F_i} \rangle =s\}\] for some $s \in \mathbb{R}$. Let now $x \in F_i$. Then, as $N_{F_i} \in L_0$, \[ \langle P_{L_0}x, N_{F_i} \rangle= \langle x, N_{F_i} \rangle =s.\] So we have that \[P_{L_0}F_i \subseteq  H_0:= H_{F_i} \cap L_0,\] where $H_0$ is an $(\ell-1)$-dimensional affine hyperplane. In addition, \[P_{L_0}x=P_{M_0}x+ s\,N_{F_i},\] which means that \[P_{L_0}F_i=P_{M_0}F_i+s\,N_{F_i}.\] Therefore, \[\vol_{\ell-1}(P_{L_0}F_i) =\vol_{\ell-1}(P_{M_0}F_i).\] For $j=1,2$ we have that \[P_{L_0}\conv\{z_j,F_i\}=\conv\{P_{L_0}z_j,P_{L_0}F_i\}.\] So $P_{L_0}\conv\{z_j,F_i\}$ is an $\ell$-dimensional pyramid, where $P_{L_0}F_i \subseteq H_0$ is its $(\ell-1)$-dimensional base and $P_{L_0}z_j$ its apex. Therefore, \[\vol_\ell(P_{L_0}\conv\{z_j,F_i\})=\frac{1}{\ell}\vol_{\ell-1}(P_{L_0}F_i)\, \dist(P_{L_0}z_j,H_0) .\] Now \[\dist(P_{L_0}z_j,H_0)= \frac{|\langle z_j,N_{F_i}\rangle-s|}{\|N_{F_i}\|_2}=\dist(z_j,H_{F_i}),\] as \[\langle P_{L_0} z_j, N_{F_i} \rangle =\langle z_j, N_{F_i}\rangle.\] Thus, 
         \begin{equation}\label{eq:pyramid-volume}    
         \vol_\ell(P_{L_0}\conv\{z_j,F_i\})=\frac{1}{\ell} \vol_{\ell-1}(P_{M_0}F_i) \frac{|\langle z_j,N_{F_i}\rangle-s|}{\|N_{F_i}\|_2}.
          \end{equation}
         Now, since $z_j \in \Delta_i^0$ and $\Phi(p) \in H_{F_i}$ we have that \begin{equation}\label{eq:inner-product-with-normal}\langle z_j,N_{F_i} \rangle -s= \langle z_j-\Phi(p), N_{F_i} \rangle=\langle (t_j-g(p))u_i,u_i-\nabla g\rangle=t_j-g(p)>0,
         \end{equation} by Claim~\ref{claim:sign-of-difference}. Therefore, by \eqref{eq:pyramid-volume} and \eqref{eq:inner-product-with-normal} \[\vol_{\ell}(P_{L_0} \conv\{z_1,F_i\}) - \vol_{\ell}(P_{L_0} \conv\{z_2,F_i\})=\frac{1}{\ell} \vol_{\ell-1}(P_{M_0}F_i) \frac{t_1-t_2}{\|N_{F_i}\|_2}.\] Now Claim~\ref{claim:ball-in-F} implies that \[\vol_{\ell-1}(P_{M_0}F_i)\geq \omega_{\ell-1}\left(\frac{ca_n}{2(d-1)}\right)^{\ell-1},\] while by applying Claim~\ref{claim:nabla-upper-bound}, we obtain \[\|N_{F_i}\|_2 =\sqrt{1+\|\nabla g\|_2^2 } \leq C_0(d),\] because $b_2=b_2(d)$ and $(h_n/a_n)_{n \in \mathbb{N}}$ is a bounded sequence. Combining with \eqref{eq:tj-difference}, we get the claim.
     \end{proof}

     For simplicity, we set \[n_{F_i}=\frac{N_{F_i}}{\|N_{F_i}\|_2}.\]

      Let also $\varepsilon=\varepsilon(d,\ell)>0$, which will be chosen sufficiently small. We define \[\mathcal{G}=\left\{ L \in G_{d,\ell}: \angle(n_{F_i},L) \leq  \frac{\varepsilon h_n}{a_n}\right\}.\] Observe that since the sequence $(h_n/a_n)_{n \in \mathbb{N}}$ is bounded, we can choose $\varepsilon>0$ sufficiently small, such that Lemma~\ref{lem:small-angles-measure} holds true for $\mathcal{G}$. In particular, \begin{equation}\label{eq:G-lower-bound}
     \nu_{d,\ell}(\mathcal{G}) \gtrsim \left(\frac{h_n}{a_n}\right)^{d-\ell}.
     \end{equation}

     \begin{claim}\label{claim:subspaces-near-NF}
         If $\varepsilon$ is chosen sufficiently small, then for every $L \in \mathcal{G}$ we have that \[\vol_{\ell}(P_{L} \conv\{z_1,F_i\}) - \vol_{\ell}(P_{L} \conv\{z_2,F_i\}) \gtrsim a_n^{\ell-1}h_n\]
     \end{claim}

     \begin{proof}
     Let $L \in \mathcal{G}$ and set \[\theta=\angle(n_{F_i},L).\] By the proof of Lemma~\ref{lem:angle-and-projection}, we can see that if \[e=\frac{P_Ln_{F_i}}{\|P_Ln_{F_i}\|_2} \in L,\] then $\angle(n_{F_i},e)=\theta$. Let also \[M:= {\rm span}\{e,n_{F_i}\}.\] If we set \[w=\frac{n_{F_i}-\cos\theta \,e}{\sin\theta},\] then $\{e,w\}$ is an orthonormal basis of $M$. We define now $R: \mathbb{R}^d \to \mathbb{R}^d$ such that \[Re=\cos\theta \, e+\sin\theta \, w=n_{F_i}, \quad Rw=-\sin \theta \, e+\cos\theta \, w,\] while \[Rx=x, \quad x \in M^\perp.\] We note that $R$ is an internal rotation of $M$ by angle $\theta$. It is clear that $R \in O(d)$. We set \[L_0 =RL,\] which is an $\ell$-dimensional subspace that contains $n_{F_i}$. We also have \[P_{L_0}-P_L=RP_LR^T-P_L=(R-I_n)P_LR^T+P_L(R^T-I_n).\] Therefore, \[d_{G_{d,\ell}}(L,L_0)=\|P_L-P_{L_0}\| \leq \|R-I_n\|+\|R^T-I_n\|.\] Notice now that if $x \in M$ \[\langle Rx,x\rangle =\cos\theta\,\|x\|_2^2,\] which implies that \[\|(R-I_n)x\|_2^2=4 \sin^2\frac{\theta}{2}\,\|x\|_2^2.\] Therefore, \[\|R-I_n\|=2\sin\frac{\theta}{2}.\] Moreover, since $R \in O(d)$ \[\|R^T-I_n\| =\|I_n-R\|=2\sin\frac{\theta}{2},\] which means that \[d_{G_{d,\ell}}(L,L_0) \leq 2\theta.\] For $j=1,2$, we have that \[\diam(\conv\{z_j,F_i\}) \lesssim a_n.\] Indeed, this follows by the fact that $\conv\{z_j,F_i\} \subseteq \Delta_i$ and \eqref{eq:simplices-diam}. Applying Lemma~\ref{lem:projection-volume-lipschitz}, we obtain \begin{equation}\label{eq:volumes-near-subspaces}|\vol_{\ell}(P_L \conv\{z_j,F_i\})-\vol_{\ell}(P_{L_0} \conv\{z_j,F_i\})| \leq C(d,\ell)\, a_n^{\ell}\,d_{G_{d,\ell}}(L,L_0) \leq C_1(d,\ell)\, a_n^\ell \,\theta.
     \end{equation} Thus, by \eqref{eq:volumes-near-subspaces} and Claim~\ref{claim:subspaces-containing-NF}
     \begin{equation*}
         \vol_{\ell}(P_{L} \conv\{z_1,F_i\}) - \vol_{\ell}(P_{L} \conv\{z_2,F_i\})\geq C_2(d,\ell\,)h_na_n^{\ell-1}-2C_1(d,\ell)\,\theta\, a_n^\ell \geq C_3(d,\ell)\,h_na_n^{\ell-1},
     \end{equation*}
         if $\varepsilon$ is chosen sufficiently small.
     \end{proof}

     We now fix $b_1$ to be large enough and $b_2,\varepsilon$ to be sufficiently small, such that all previous Claims and Lemmas hold.

\begin{claim}\label{claim:difference-integral-functionals}
   We have that \[V_\ell(\conv\{z_1,F_i\})-V_\ell(\conv\{z_2,F_i\}) \gtrsim a_n^{2\ell-d-1}h_n^{d-\ell+1}.
    \]
\end{claim}

\begin{proof}
Using Kubota's formula from \eqref{eq:intrinsic-kubota-formula} and Claim~\ref{claim:convex-hull-inclusion}, we have that 
\begin{align*}
&V_\ell(\conv\{z_1,F_i\})-V_\ell(\conv\{z_2,F_i\})\\
        &\hspace*{1cm}= c_{d,\ell}\int_{G_{d,\ell}} \vol_\ell(P_L\conv\{z_1,F_i\})-\vol_\ell(P_L\conv\{z_2,F_i\})\, d\nu_{d,\ell}(L)\\
        &\hspace*{1cm}\gtrsim \int_{\mathcal{G}} \vol_\ell(P_L\conv\{z_1,F_i\})-\vol_\ell(P_L\conv\{z_2,F_i\})\, d\nu_{d,\ell}(L).       
    \end{align*} 
Applying Claim~\ref{claim:subspaces-near-NF} and \eqref{eq:G-lower-bound}, concludes the proof.
\end{proof}

Let now $Z\sim \mu_{\Delta_i^0}$. Then by \eqref{eq:delta0-alternative-writing} we can write \[Z=y_i+P+Tu_i,\] where \[P=P_{u_i^\perp}(Z-y_i), \quad T=\langle Z-y_i,u_i\rangle\] are the tangential and radial components respectively. We condition on \[P=p \in \frac{b_2}{2}W.\]
Therefore, as we have seen, the possible values of $T$ form the interval \[J(p)=[(1-b_2)h_n,(1-p_W(p))\,h_n].\] The conditional density of $T$, given that $P=p$ is \[f_p(t)=\frac{f_\mu(y_i+p+tu_i)}{\int_{J(p)}f_\mu(y_i+p+su_i)\, ds} \mathds{1}_{J(p)}(t).\] So, if $I \subseteq J(p)$, in view of \eqref{eq:simplices-f-asymptotics}, we have that \[\mathbb{P}(T \in I\mid P=p) \geq \frac{e^{-1}f_\mu(r_n)|I|}{f_\mu(r_n)|J(p)|}=e^{-1}\frac{|I|}{|J(p)|}.\] We remind that $I_1(p),I_2(p) \subseteq J(p)$ and \[|I_j(p)|=\frac{b_2}{8}h_n, \quad j=1,2,\] as well as \[|J(p)|=(b_2-p_W(p))h_n \leq b_2h_n.\] Hence \begin{equation}\label{eq:probability-of-interval}
\mathbb{P}(T \in I_j(p)\mid P=p) \geq \frac{e^{-1}}{8}.
\end{equation}
For fixed $p \in \frac{b_2}{2}W$, we set \[g_p(t)=V_\ell(\conv\{y_i+p+tu_i,F_i\}).\] By Claim~\ref{claim:difference-integral-functionals}, if $t_1 \in I_1(p)$ and $t_2 \in I_2(p)$, we get \[|g_p(t_1)-g_p(t_2)|^2 \gtrsim \delta_{n,d,\ell}^2,\] where \[ \delta_{n,d,\ell}=a_n^{2\ell-d-1} h_n^{d-\ell+1}\] Let also $T'$ be an independent copy of $T$, conditionally on $P=p$. Thus,
\begin{equation}\label{eq:conditional-variance}
\begin{aligned}
        \Var\bigl(g_p(T)\mid P=p\bigr)&=\frac{1}{2} \mathbb{E}\!\left [\bigl(g_p(T)-g_p(T')\bigr)^2\mid P=p\right] \geq \\
    &\geq \frac{1}{2}\delta_{n,d,\ell}^2\, \mathbb{P}\bigl(T \in I_1(p), T' \in I_2(p)\mid P=p\bigr) \geq\\
    &\geq  C \,\delta_{n,d,\ell}^2,
\end{aligned}
\end{equation}
where we also used \eqref{eq:probability-of-interval}. We will also need the following:

\begin{claim}\label{claim:W-positive-probability}
There exists a constant $c_d>0$ such that \[\mathbb{P} \left(P\in \frac{b_2}{2}W\right) \geq c_d. \]    
\end{claim}

\begin{proof}
     Set \[S=\biggl\{y_i+p+tu_i: p \in \frac{b_2}{2}W, t \in J(p)\biggr\}\] and notice that $P\in \frac{b_2}{2}W$ if and only if $Z \in S$. Consequently, using \eqref{eq:simplices-f-asymptotics}  again,\[\mathbb{P} \left(P\in \frac{b_2}{2}W\right)=\mathbb{P}(Z \in S)=\frac{\mu(S)}{\mu(\Delta_i^0)} \geq e^{-1}\frac{\vol_d(S)}{\vol_d(\Delta_i^0)}.\] By Fubini we have \[
     \vol_d(S)=\int_{\frac{b_2}{2}W} |J(p)| \, dp \geq \left(\frac{b_2}{2} \right)^d\vol_{d-1}(W)h_n.  
     \] Now, since \[\frac{ca_n}{d-1}B_{u_i^\perp} \subseteq W,\] we get \[\vol_{d}(S) \geq C_1(d) a_n^{d-1}h_n.\] Combining with \eqref{eq:simplices-volume-asymptotics} concludes the proof.
\end{proof}

 Set now \[Y= V_\ell(\conv\{Z,F_i\}).\] The law of total variance gives \begin{align*}
 \Var(Y) &=\mathbb{E}\bigl[\Var(Y \mid P)\bigr]+\Var\bigl(\mathbb{E}[Y\mid P]\bigr) \\
 &\geq \mathbb{E}\bigl[\Var(Y\mid P)\bigr] \geq \mathbb{E}\bigl[\Var(Y\mid P) \, \mathds{1}_{\{P \in \frac{b_2}{2}W\}}\bigr] \\
 &\geq C\, \delta_{n,d,\ell}^2\, \mathbb{P}\left(P \in \frac{b_2}{2}W \right),
 \end{align*} where we used \eqref{eq:conditional-variance}. Apply Claim~\ref{claim:W-positive-probability} and the proof is complete.
\end{proof}

So we have now established the main ingredient of the proof. However, before we proceed, we need some more geometric lemmas.

\begin{lemma}\label{lem:summability-possible}
    Let $R_N=\{x_1,\ldots,x_N\} \subseteq \mathbb{R}^d$ and denote by $K_N$ the convex hull of $R_N$. Let also $i \in \{1,\ldots,m\}$. Assume that exactly one point of $R_N$, say $z_i^j$ lies in $\Delta_i^j, \, j=0,\ldots,d$ and no further point lies in $\widetilde{H}_i$. If $F_i=\conv\{z_i^1,\ldots,z_i^d\}$ and $Q_i=\conv(R_N\setminus\{z_i^0\})$, then \[K_N=Q_i \cup \conv\{z_i^0,F_i\}, \quad Q_i \cap \conv\{z_i^0,F_i\}=F_i. \]
\end{lemma}

\begin{proof}
We will first show that \begin{equation}\label{eq:shadow-inclusion}  
Q_i \subseteq S(z_i^0,F_i)
\end{equation}
By the convexity of $S(z_i^0,F_i)$ it is enough to show that every generating point $x$ of $Q_i$ belongs to $S_i(z_i^0,F_i)$. If $x=z_i^j$ for some $j=1,\ldots,d$, then \[x \in F_i \subseteq S(z_i^0,F_i).\] Assume otherwise. Then $x \in \mathbb{R}^d \setminus \widetilde{H}_i$ and Lemma~\ref{lem:shadow-lemma} implies that $x \in S(z_i^0,F_i)$.

 Now, about the equalities of the lemma, it is clear that \[Q_i \cup \conv\{z_i^0,F_i\}\subseteq K_N,\quad F_i \subseteq Q_i \cap \conv\{z_i^0,F_i\}.\]  Let  $x \in K_N$, which means that $x \in [z,z_i^0]$ for some $ z \in Q_i$. It suffices to prove that \[[z,z_i^0] \subseteq Q_i \cup \conv\{z_i^0,F_i\}.\] By \eqref{eq:shadow-inclusion}, we get that there exists some $v \in F_i$, 
 such that $v \in [z,z_i^0]$. Therefore, \[[z,z_i^0]=[z,v]\cup [v,z_i^0] \subseteq Q_i \cup \conv\{z_i^0,F_i\},\] as $F_i \subseteq Q_i$. Hence, the first equality is proven. For the second one, let $x \in Q_i \cap \conv\{z_i^0,F_i\}$. If $\beta(x) \in \mathbb{R}^{d+1}$ is the barycentric vector of $x$, with respect to  $z_i^0,\ldots,z_i^d$, then the fact that $x \in Q_i$, along with \eqref{eq:shadow-inclusion} and \eqref{eq:shadow-barycentric}, implies that \[\beta_0(x) \leq 0.\] However, $\beta_0(x) \geq 0$, as $x \in\conv\{z_i^0,F_i\}$. Therefore, $\beta_0(x)=0$, which means that $x \in F_i$. 
 \end{proof}

\begin{corollary}\label{cor:intrinsic-induction}
Let $R_N=\{x_1,\ldots,x_N\} \subseteq \mathbb{R}^d$ and denote by $K_N$ their convex hull. Let also $i \in \{1,\ldots,m\}$. Assume that exactly one of these points, say $z_i^j$ lies in $\Delta_i^j, \, j=1,\ldots,d$ and no further point lies in $\widetilde{H}_i \cup \Delta_i^0$. If $F_i=\conv\{z_i^1,\ldots,z_i^d\}$ and $z_1,z_2 \in \Delta_i^0$, then we have that
\[V_\ell(\conv\{R_N,z_1\})-V_\ell(\conv\{R_N,z_2\})=V_\ell(\conv\{z_1,F_i\})-V_\ell(\conv\{z_2,F_i\}).\]
\end{corollary}

\begin{proof}
    Applying twice Lemma~\ref{lem:summability-possible} for the samples $z_k, x_1,\ldots,x_N, \ k=1,2$ and using \eqref{eq:intrinsic-valuation} we obtain:
    \[V_\ell(\conv\{R_N,z_k\})= V_\ell(K_N)+V_\ell(\conv\{z_k,F_i\})-V_\ell(F_i), \quad k=1,2.
    \] Subtracting the two equations, concludes the proof.
\end{proof}

\begin{proof}[Proof of Theorem~\ref{thm:main-general-radial}]
Let $R_n=\{X_1,\ldots,X_n\}$ denote our random sample of $n$ points. Set \[\widetilde{n}_0= \max\{d+1,n_0(\mu,d)\},\] where $n_0(\mu,d)$ is the one from Lemma~\ref{lem:Hij-measure-upper-bound}. Suppose that $n \geq \widetilde{n}_0$. For $1 \leq i \leq m$ denote by $A_i$ the event that exactly one random point of $R_n$ is contained in each simplex $\Delta_i^j, \, j=0,\ldots,d$ and no further point of $R_n$ is contained in $\widetilde{H}_i$.

\begin{claim}\label{claim:A-events-lower-bound}
    There exists $c_0(\mu,d)>0$ such that for every $i=1,\ldots,m$ \[\mathbb{P}(A_i) \geq c_0(\mu,d).\]
\end{claim}

\begin{proof}
Notice that by its definition, we have that
\[
\mathbb P(A_i)
=
\frac{n!}{(n-d-1)!}
\left(1-\mu(\widetilde{H}_i)\right)^{n-d-1} \prod_{j=0}^d \mu(\Delta_i^j) \] Hence,  by Lemma~\ref{lem:simplex-mass}, Lemma~\ref{lem:choice-tn} Lemma~\ref{lem:H-plus-upper-bound}, Lemma~\ref{lem:Hij-measure-upper-bound}
\[
\mathbb{P}(A_i) \geq C_1(d)\,c_{\mu,d}\,  n^{d+1}\frac1{n^{d+1}}
\left(1-\frac {C_2(d)}{n}\right)^{n-d-1}
\ge c_0(\mu,d),
\]
where $c_0(\mu,d)>0$.
\end{proof}

For every $i=1,\ldots,m$, define \[J_i:=\begin{cases}
    \{k\} &, \text{ if } A_i \text{ occurs and }  X_k \text{ is the unique point in } \Delta_i^0\\
    \emptyset &, \text{ if } A_i^c \text{ occurs}. 
\end{cases}
\] Now we define \[\mathcal{J}=\bigcup_{i=1}^{m} J_i \subseteq \{1,\ldots,n\} \] and introduce the ``masked" variables \[
Y_k=\begin{cases}
x_\ast &, k \in \mathcal{J}\\
X_k &, k \notin \mathcal{J}
\end{cases}
\quad k=1,\ldots,n,\] where $x_\ast \notin \mathbb{R}^d$. With this notation we can set \[\mathcal{F}=\sigma(J_1,\ldots,J_n,Y_1,\ldots,Y_n).\] Let also \[I = \{1 \leq i \leq m: J_i \neq \emptyset\}=\{1 \leq i \leq m: A_i \text{ occurs}\}.\] Conditioning on $\mathcal{F}$ means that $I$ is completely determined and so are all the points of $R_n$, except for the unique points in $\Delta_i^0$, say $Z_i^0$, for which $\mathds{1}_{A_i}=1$. Let now $i \in I$ and set \[D_i=R_n \setminus\{Z_i^0\}.\] Choose some $z_i^\ast \in \Delta_i^0$. Then, by Corollary~\ref{cor:intrinsic-induction},
\begin{align*}
V_\ell(K_n^\mu)&=V_\ell(\conv\{D_i,Z_i^0\})\\
&=V_\ell(\conv\{D_i,z_i^\ast\})-V_\ell(\conv\{z_i^\ast,F_i\})+V_\ell(\conv\{Z_i^0,F_i\})\\
&=C_i(\mathcal{F})+V_\ell(\conv\{D_i,z_i^\ast\})+V_\ell(\conv\{Z_i^0,F_i\}),
\end{align*} as $F_i$ is deterministic, conditioning on $\mathcal{F}$. Moreover, the polytope  $\conv\{D_i,z_i^\ast\}$ has one fewer random generating point than $K_n^\mu$, with respect to $\mathcal{F}$. Repeating the process, we obtain \[V_\ell(K_n^\mu)=C(\mathcal{F})+\sum_{i \in I} V_\ell(\conv\{Z_i^0,F_i\}). \] Corollary~\ref{cor:independence-under-condition} ensures that under $\mathcal{F}$, the random variables $Z_i^0, \, i \in I$ are independent, because for $i,k \in I, \, i \neq k$, the occurence of $A_k$ does not impose further restriction on the location of $Z_i^0 \in \Delta_i^0$. Hence, \[\Var\bigl(V_\ell(K_n^\mu)\mid \mathcal{F}\bigr)=\sum_{i \in I} \Var\bigl(V_\ell(\conv\{Z_i^0,F_i\}\bigr)\mid \mathcal{F}), \] where \[Z_i^0 \mid \,\mathcal{F} \sim \mu_{\Delta_i^0}.\] Applying Proposition~\ref{prop:variance-lower-bound} yields \[\Var\bigl(V_\ell(K_n^\mu)\mid \mathcal{F}\bigr) \gtrsim a_n^{4\ell-2d-2}h_n^{2d-2\ell+2} \sum_{i=1}^m\mathds{1}_{A_i}.\] Thus, the law of total variance gives
\begin{align*}
    \Var\bigl(V_\ell(K_n^\mu)\bigr)&= \mathbb{E}\bigl[\Var\bigl(V_\ell(K_n^\mu)\mid \mathcal{F}\bigr)\bigr]+ \Var\bigl(\mathbb{E}\bigl[V_\ell(K_n^\mu)\mid\mathcal{F}\bigr]\bigr)\\
    &\geq \mathbb{E}\bigl[\Var\bigl(V_\ell(K_n^\mu)\mid \mathcal{F}\bigl)\bigr]\gtrsim a_n^{4\ell-2d-2}h_n^{2d-2\ell+2} \sum_{i=1}^{m}\mathbb{P}(A_i)\\
    &\gtrsim_\mu a_n^{4\ell-2d-2}h_n^{2d-2\ell+2}\,m \gtrsim  r_n^{2\ell-\frac{d+3}{2}}h_n^{\frac{d+3}{2}},
\end{align*}
    where we used Claim~\ref{claim:A-events-lower-bound} and \eqref{eq:m-estimate}. To conclude the proof, we also need to lower-bound \[\min_{d+1 \leq n \leq\,\widetilde{n}_0} \Var\bigl(V_\ell(K_n^\mu)\bigl).\] The following claim is enough.
    \begin{claim}
        For every $n \geq d+1$, \[\Var\bigl(V_\ell(K_n^\mu)\bigl)>0.\]
    \end{claim}

    \begin{proof}
        Let $ n \geq d+1$ and set \[E_1=\bigcap_{k=1}^{n}\big\{X_k \in B_2^d\big\}.\] Then $\mathbb{P}(E_1)>0$, as $\mu$ has full support and on $E_1$ we have that $K_n^\mu \subseteq B_2^d$. Thus, \[V_\ell(K_n^\mu) \leq V_\ell(B_2^d):=\alpha.\] Let now $P_d=\conv\{p_0,\ldots,p_d\}$ be a $d$-simplex with $V_\ell(P_d)>4\alpha.$ Then by continuity of the intrinsic volumes with respect to Hausdorff metric, we can find $\delta>0$ such that if \[q_k \in B(p_k,\delta),\quad\quad k=0,\ldots,d ,\] then $V_\ell(\conv\{q_0,\ldots,q_d\})>2\alpha$. Hence, if we set \[E_2=\bigcap_{k=1}^{d+1}\big\{X_k \in B(p_{k-1},\delta)\big\},\] we have that $\mathbb{P}(E_2)>0$, as $\mu$ has full support and on $E_2$ \[V_\ell(K_n^\mu)>2\alpha.\] The proof is complete, because with positive probability, $V_\ell(K_n^\mu)$ is not constant.
    \end{proof}
\end{proof}

\section{Lower bound for the number of faces}

In this section, we prove Theorem~\ref{thm:number-of faces}. The method is quite similar. In fact, we first establish Proposition~\ref{prop:partial-faces-lower-bound} which is the counterpart of Proposition~\ref{prop:variance-lower-bound}. We fix the base simplex $F_i$ and compare two positive-probability configurations of two random points in $\Delta_i^0$. A combinatorial analysis, based on stacking and the face structure of simplices, shows that these configurations produce different numbers of $\ell$-faces, providing the necessary local variance contribution.

Fix $0 \leq \ell \leq d-1$. Let also $z_i^j \in \Delta_i^j$ for every $j=1,\ldots,d$ fixed and set \[F_i=\conv\{z_i^1,\ldots,z_i^d\}.\] 
We remind that $b_2$ is small enough, depending only on $d$ such that all our arguments hold.

\begin{proposition}\label{prop:partial-faces-lower-bound}
    If $Z_1,Z_2$ are independent random vectors, distributed according to $\mu_{\Delta_i^0}$, then \[\Var\bigl(f_\ell(\conv\{Z_1,Z_2,F_i\})\bigr) \gtrsim1.\] 
\end{proposition}

\begin{proof}
The proof will be based on some claims. We first set  \[\eta=\frac{1}{2(d+1)}, \quad \delta=\frac{b_2\eta}{8d}.\]If $\lambda(x)$ is the barycentric vector of some $x \in \mathbb{R}^d$ with respect to $\Delta_i$, we also set  \[K_i=\{x \in \Delta_i:\lambda_j(x) \geq \eta, \ j=0,\ldots,d\}, \quad \mathcal{L}_i=y_i^0+b_2(K_i-y_i^0).\]
As $K_i \subseteq \Delta_i$, we get that $\mathcal{L}_i\subseteq \Delta_i^0$. If $\tilde{y}_i=\frac{1}{d+1}\sum_{j=0}^d y_i^j$ is the centroid of $\Delta_i$ then it is easy to see that $K_i=\tilde{y}_i+\frac{1}{2}(\Delta_i-\tilde{y}_i)$. Therefore, \begin{equation}\label{eq:L-volume-comparable}\vol_d(\mathcal{L}_i)=2^{-d}\vol_d(\Delta_i^0).\end{equation} Fix some $z_i^0 \in \mathcal{L}_i$ and set $B_i=\conv\{z_i^0,\ldots,z_i^d\}$. Observe that \[\lambda_0(z_i^0) \geq1-b_2+b_2\eta, \quad \lambda_j(z_i^0) \geq b_2\eta, \quad j=1,\ldots,d.\] Define now \[U_{\rm in}(z_i^0,F_i)=\{x:\delta<\beta_j(x)<2\delta, \ j=1,\ldots, d\},\] where $\beta(x)$ is the barycentric vector of $x$ with respect to $Z_i$. 

\begin{claim}\label{claim:Uin-subset}
    We have that \[ U_{\rm in}(z_i^0,F_i) \subseteq \Delta_i^0.\]
\end{claim}

\begin{proof}
Let $x \in U_{\rm in}(z_i^0,F_i)$. Then, we can write \[x=\sum_{k=0}^{d}\beta_k(x)z_i^k=z_i^0+\sum_{k=1}^d\beta_k(x)(z_i^k-z_i^0).\] Therefore, since barycentric coordinates are affine functions, for every $j=0,\ldots,d$
\[
    |\lambda_j(x)-\lambda_j(z_i^0)|=\biggl|\sum_{k=1}^{d}\beta_k(x)(\lambda_j(z_i^k)-\lambda_j(z_i^0))\biggr| \leq \sum_{k=1}^d|\beta_k(x)| \leq \frac{b_2\eta}{4}.
\] Hence,
\[\lambda_0(x) \geq 1-b_2+\frac{3b_2\eta}{4} \geq 1-b_2, \quad \lambda_j(x) \geq \frac{3b_2}{4}\geq 0, \quad j=1,\ldots,d,\] which implies that $x \in \Delta_i^0$.   
\end{proof}

Set also \[U_{\rm st}(z_i^0,F_i)=\{x:\beta_1(x) \in (-2\delta,-\delta), \, \beta_j(x) \in (\delta,2\delta), \ j=2,\ldots,d\}.\] The exact argument as in Claim~\ref{claim:Uin-subset} yields: 
    \begin{equation}\label{eq:Ust-subset}
        U_{\rm st}(z_i^0,F_i) \subseteq \Delta_i^0.
    \end{equation}

We will now lower-bound the measures of these sets.

\begin{claim}\label{claim:lower-bound-in-stacked}
        There exists a constant $c(d)>0$ such that \[\mu_{\Delta_i^0}\bigl(U_{\rm in}(z_i^0,F_i)\bigr),\,\mu_{\Delta_i^0}\bigl(U_{\rm st}(z_i^0,F_i)\bigl) \geq c(d).\]
\end{claim}

\begin{proof}
    We will prove it only for $U_{\rm in}(z_i^0,F_i)$, as the other part is quite similar. Using \eqref{eq:simplices-f-asymptotics} we get \[\mu_{\Delta_i^0}\bigl(U_{\rm in}(z_i^0,F_i)\bigr) \asymp \frac{\vol_d\bigl(U_{\rm in}(z_i^0,F_i)\bigr)}{\vol_d(\Delta_i^0)}.\] Since every $x \in \mathbb{R}^d$ can be written uniquely as \[x=z_i^0+\sum_{j=1}^{d}\beta_j(z_i^j-z_i^0),\] where $\beta_j=\beta_j(x)$ for every $j=1,\ldots,d$, if we set \[U_{\rm in}=\{\beta \in \mathbb{R}^d: \beta_j \in (\delta,2\delta), \ j=1,\ldots,d\},\] then \begin{align*}
    \vol_d\bigl(U_{\rm in}(z_i^0,F_i)\bigr)&=|{\rm det}(z_i^1-z_i^0,\ldots,z_i^d-z_i^0)\,|\, \vol_d(U_{\rm in})\\
    &=d!\,\vol_d(B_i)\,\delta^d\\
    &=d! \,|{\rm det(A)|}\,\vol_d(\Delta_i)\,\delta^d,
    \end{align*} where $A$ is the matrix in the proof of Lemma~\ref{lem:geometric-small-c}. As $A=I_{d+1}+E$, where \[\|E:\ell_1^{d+1} \to \ell_1^{d+1}\| \leq 2b_2,\] we can show that every eigenvalue $\kappa$ of $A$, satisfies $|\kappa| \geq 1-2b_2$. Thus, \[|{\rm det}(A)| \geq (1-2b_2)^{d+1},\] which completes the proof, as $b_2,\delta$ are positive functions of $d$.
\end{proof}

\begin{claim}\label{claim:intermediate-variance}
    If we set $Y_i=f_\ell(\conv\{Z_1,z_i^0,F_i\})$, then \[\Var(Y_i) \gtrsim 1.\]
\end{claim}

\begin{proof}
     To begin with, notice that \[Y_i=f_\ell(\conv\{Z_1,B_i\})\]By construction, if $x \in U_{\rm in}(z_i^0,F_i)$, then $x \in {\rm int}(B_i)$. Consequently, $\conv\{x,B_i\}=B_i$ and \begin{equation}\label{eq:faces-in}       
    f_\ell(\conv\{x,B_i\})=f_\ell(B_i), \quad x \in U_{\rm in}(z_i^0,F_i).
    \end{equation}
    Now suppose $x \in U_{\rm st}(z_i^0,F_i)$. By constuction, \[\beta_1(x)<0, \quad \beta_j(x)>0, \quad j=0,2,\ldots,d.\] Consider the facet $G_1=\conv\{z_i^0,z_i^2,\ldots,z_i^d\}$ of $B_i$. Then $x$ lies beyond $G_1$ and beneath every other facet of $B_i$. Hence, $\conv\{x,B_i\}$ is obtained from $B_i$ by stacking a pyramid onto the facet $G_1$; see for example \cite[\S 2.4]{Ziegler}. Therefore, using also \eqref{eq:m-simplex-faces},
    \[f_{d-1}(\conv\{x,B_i\})=f_{d-1}(B_i)+d-1, \quad f_{\ell}(\conv\{x,B_i\})=f_{\ell}(B_i)+\binom{d}{\ell}, \quad \ell=0,\ldots,d-2,\] which can also be written as \begin{equation}\label{eq:faces-stacked}f_{\ell}(\conv\{x,B_i\})=f_\ell(B_i)+\tau_{d,\ell}, 
    \end{equation}
    where $\tau_{d,\ell} \geq 1$ for every $\ell=0,\ldots,d-1$. Let now $\widetilde{Z}_1$ be an independent copy of $Z_1$ and set \[ \widetilde{Y}_i=f_\ell(\conv\{\widetilde{Z}_1,z_i^0,F_i\})=f_\ell(\conv\{\widetilde{Z}_1B_i\}),\] which is an independent copy of $Y_i$. Then we have 
        \[
            \Var(Y_i)=\frac{1}{2}\mathbb{E}\bigl[(Y_i-\widetilde{Y}_i)^2\bigr].
        \] If we set \[\Pi=\bigl\{Z_1 \in U_{\rm in}(z_i^0,F_i), \, \widetilde{Z}_1 \in  U_{\rm st}(z_i^0,F_i)\bigr\} ,\] then on $\Pi$ , by \eqref{eq:faces-in} and \eqref{eq:faces-stacked}, we have \[|Y_i -\widetilde{Y}_i|=\tau_{d,\ell} \geq 1\] Therefore, as $Z_1,\widetilde{Z}_1$ are i.i.d random vectors, distributed according to $\mu_{\Delta_i^0}$ \[
            \Var(Y_i) \geq \frac{1}{2}\,\mu_{\Delta_i^0}\bigl(U_{\rm in}(z_i^0,F_i^0)\bigr)\,\mu_{\Delta_i^0}\bigl(U_{\rm st}(z_i^0,F_i^0)\bigr)\,\tau_{d,\ell}^2 \gtrsim 1,
        \] by Claim~\ref{claim:lower-bound-in-stacked}.  
\end{proof}

Now, we are able to finish the proof of Proposition~\ref{prop:partial-faces-lower-bound}. Indeed,  by the law of total variance, and the independence of $Z_1$ and $Z_2$, \begin{align*}
        \Var \bigl(f_\ell(\conv\{Z_1,Z_2,F_i\})\bigr)&\geq \mathbb{E}\bigl[\Var\bigl(f_\ell(\conv\{Z_1,Z_2,F_i\})\mid Z_2\bigr)\,\bigr]\\
        &\geq\int_{\mathcal{L}_i} \Var\bigl(f_\ell(\conv\{Z_1,z_i^0,F_i\})\bigr)\, d\mu_{\Delta_i^0}(z_i^0)\\
        &\gtrsim\mu_{\Delta_i^0}(\mathcal{L}_i)  \asymp \frac{\vol_d(\mathcal{L}_i)}{\vol_d(\Delta_i^0)}\asymp 1,
        \end{align*} where we used Claim~\ref{claim:intermediate-variance}, \eqref{eq:simplices-f-asymptotics} and \eqref{eq:L-volume-comparable}
\end{proof}

Before we proceed to the proof of the theorem, we will need some geometric lemmas, corresponding to the ones for the intrinsic volumes.

\begin{lemma}\label{lem:faces-construction}
Let $R_N=\{x_1,\ldots,x_N\} \subseteq \mathbb{R}^d$ and denote by $K_N$ the convex hull of $R_n$. Let also $i \in \{1,\ldots,m\}$. Assume that exactly one point of $R_N$, say $z_i^j$ lies in $\Delta_i^j, \, j=1,\ldots,d$, exactly two points of $R_N$, say $z_1,z_2$ lie on $\Delta_i^0$ and no further point lies in $\widetilde{H}_i$.  Set \[F_i=\conv\{z_i^1,\ldots,z_i^d\}, \quad Q_i=\conv(R_N\setminus\{z_1,z_2\}), \quad P_i(z_1,z_2)=\conv\{z_1,z_2,F_i\}. \] Then $Q_i$ and $P_i(z_1,z_2)$ lie on opposite sides of $H_{F_i}={\rm aff}(F_i)$ and\[K_N=Q_i \cup P_i(z_1,z_2), \quad Q_i \cap P_i(z_1,z_2)=F_i.\]   
\end{lemma}

\begin{proof}
For $x \in \mathbb{R}^d$ let $\lambda(x)$ and $\beta(x)$  be the barycentric vectors with respect to $y_i^0,\ldots,y_i^d$ and $z_1,z_i^1,\ldots,z_i^d$ respectively. Then we have that \[H_{F_i}=\{x:\beta_0(x)=0\}.\] If $x \in Q_i$, then by \eqref{eq:shadow-inclusion} for $z_1$, there exists $f \in F_i$ such that $x=(1-t)z_1+tf$ for some $t \geq 1.$ Therefore, \[\beta_0(x)=(1-t)\beta(z)+t\beta(f)=1-t \leq 0.\] We will now show that $\beta_0(x) \geq 0$ for every $x \in P_i(z_1,z_2)$. It suffices to show that $\beta_0(z_2)>0$. Indeed, with the notation of the proof of Lemma~\ref{lem:geometric-small-c} \begin{align*}
    \beta_0(z_2) &\geq \lambda_0(z_2)-\|\beta(z_2)-\lambda(z_2)\|_1 = \lambda_0(z_2)- \|(A^{-1}-I_{d+1})\lambda (z_2)\|_1\\
    &\geq \lambda_0(z_2)-\frac{2b_2}{1-b_2}\|\lambda(z_2)\|_1 \geq 1-b_2-\frac{2b_2}{1-2b_2}>0,
\end{align*} if $b_2$ is small enough, where we also used the fact that $z_2 \in \Delta_i^0$.

For the set equalities, by our previous work it is clear that $Q_i \cap P_i(z_1,z_2)=F_i$. So it remains to show that \[K_N = Q_i \cup P_i(z_1,z_2).\] It is enough to prove that $Q_i \cup P_i(z_1,z_2)$ is convex. Indeed, let $q \in Q_i$ and $p \in P_i(z_1,z_2)$. If $p \in F_i$, then $p \in Q_i$ and therefore $[q,p]\subseteq Q_i $. So suppose $p \notin F_i$ and write \begin{equation}\label{eq:easy-equation}
p=az_1+bz_2+\gamma f,
\end{equation}
as a convex combination, where $f \in F_i$. We shall prove that there exists $r \in [q,p] \cap F_i$. Then \[[q,p]=[q,r] \cup [r,p] \subseteq Q_i \cup P_i(z_1,z_2). \]

We have that \[\beta_0(q) \leq 0, \quad \beta_0(p)=a+b\beta_0(z_2).\] Therefore, $[q,p]$ intersects $H_{F_i}$ at \begin{equation}\label{eq:complicated-equation-1}r=\frac{a+b\beta_0(z_2)}{a+b\beta_0(z_2)-\beta_0(q)}\,q-\frac{\beta_0(q)}{a+b\beta_0(z_2)-\beta_0(q)}\,p.\end{equation}
    We need to show that $r \in F_i$. Applying \eqref{eq:shadow-inclusion} twice for $z_1$ and $z_2$, we deduce that \[q \in S(z_1,F_i) \cap S(z_2,F_i).\] Hence, we can find $f_1,f_2 \in F_i$ such that \[f_1 \in [q,z_1], \quad f_2 \in [q,z_2].\] A direct calculation shows that \begin{equation}\label{eq:complicated-equation-2}
        f_1=\frac{1}{1-\beta_0(q)}\,q-\frac{\beta_0(q)}{1-\beta_0(q)}\,z_1, \quad f_2=\frac{\beta_0(z_2)}{\beta_0(z_2)-\beta_0(q)}\,q-\frac{\beta_0(q)}{\beta_0(z_2)-\beta_0(q)}\,z_2,\end{equation} as convex combinations. Combinining \eqref{eq:easy-equation}, \eqref{eq:complicated-equation-1} and \eqref{eq:complicated-equation-2} we obtain, \[r=\frac{a\bigl(1-\beta_0(q)\bigr)}{a+b\beta_0(z_2)-\beta_0(q)}\,f_1+\frac{b\bigl(\beta_0(z_2)-\beta_0(q)\bigr)}{a+b\beta_0(z_2)-\beta_0(q)}\,f_2-\frac{\gamma\beta_0(q)}{a+b\beta_0(z_2)-\beta_0(q)}\,f \in F_i,\] as a convex combination.
\end{proof}

With the notation of Lemma~\ref{lem:faces-construction}, $K_N$ is the connected sum of $Q_i$ and $P_i(z_1,z_2)$ along the common facet $F_i$. The number of faces of $K_N$ is given by standard arguments from convex geometry (see also \cite[Remark~3.2.2]{Rich-Geb}):

\begin{corollary}\label{cor:relations-of-faces}
\leavevmode
    \begin{enumerate}
    \item[{\rm (i)}] If ${\rm dim}(Q_i)=d-1$, then \[f_{d-1}(K_N)=f_{d-1}(P_i(z_1,z_2)).\]
    \item[{\rm (ii)}]If ${\rm dim}(Q_i)=d$, then \[f_{d-1}(K_N)=f_{d-1}(Q_i)+f_{d-1}(P_i(z_1,z_2))-2.\] 
    \item[{\rm (iii)}] For every $0 \leq \ell \leq d-2$, \[f_\ell(K_N)=f_{\ell}(Q_i)+f_{\ell}(P_i(z_1,z_2))-f_\ell(F_i).\]
    \end{enumerate}
\end{corollary}

We are now in position to prove Theorem~\ref{thm:number-of faces}.

\begin{proof}[Proof of Theorem~\ref{thm:number-of faces}]
Let $R_n=\{X_1,\ldots,X_n\}$ denote our random sample of $n$ points. Set \[\widetilde{n}_0= \max\{d+1,n_0(\mu,d)\},\] where $n_0(\mu,d)$ is the one from Lemma~\ref{lem:Hij-measure-upper-bound}. Suppose that $n \geq \widetilde{n}_0$. For $1 \leq i \leq m$ denote by $A_i$ the event that exactly one random point of $R_n$ is contained in each simplex $\Delta_i^j, \, j=1,\ldots,d$, exactly two points are contained in $\Delta_i^0$ and no further point of $R_n$ is contained in $\widetilde{H}_i$.

\begin{claim}\label{claim:A-face-events-lower-bound}
    There exists $c_0(\mu,d)>0$ such that for every $i=1,\ldots,m$ \[\mathbb{P}(A_i) \geq c_0(\mu,d).\]
\end{claim}

\begin{proof}
Notice that by its definition, we have that
\[
\mathbb P(A_i)
=
\frac{n!}{(n-d-2)!}
\left(1-\mu(\widetilde{H}_i)\right)^{n-d-2}\mu(\Delta_i^0)\prod_{j=0}^d \mu(\Delta_i^j)
 \] Hence,  by Lemma~\ref{lem:simplex-mass}, Lemma~\ref{lem:choice-tn}, Lemma~\ref{lem:H-plus-upper-bound} and Lemma~\ref{lem:Hij-measure-upper-bound}
\[
\mathbb{P}(A_i) \geq C_1(d)\,c_{\mu,d}  \,n^{d+2}\frac1{n^{d+2}}
\left(1-\frac {C_2(d)}{n}\right)^{n-d-2}
\ge c_0(\mu,d),
\]
where $c_0(\mu,d)>0$.
\end{proof}

For every $i=1,\ldots,m$, define \[J_i:=\begin{cases}
    \{k_1,k_2\} &, \text{ if } A_i \text{ occurs and }  X_{k_1},X_{k_2} \text{ are the unique points in } \Delta_i^0\\
    \emptyset &, \text{ if } A_i^c \text{ occurs}. 
\end{cases}
\] Now we define \[\mathcal{J}=\bigcup_{i=1}^m  J_i\subseteq \{1,\ldots,n\} \] and introduce the ``masked" variables \[
Y_k=\begin{cases}
x_\ast &, k \in \mathcal{J}\\
X_k &, k \notin \mathcal{J}
\end{cases}
\quad k=1,\ldots,n,\] where $x_\ast \notin \mathbb{R}^d$. With this notation we can set \[\mathcal{F}=\sigma(J_1,\ldots,J_n,Y_1,\ldots,Y_n).\] Let also \[I = \{1 \leq i \leq m: J_i \neq \emptyset\}=\{1 \leq i \leq m: A_i \text{ occurs}\}.\] Conditioning on $\mathcal{F}$ means that $I$ is completely determined and so are all the points of $R_n$, except for the two unique points in $\Delta_i^0$, say $Z_1^{(i)},Z_2^{(i)}$, for which $\mathds{1}_{A_i}=1$. Let now $i \in I$. With the notation of Lemma~\ref{lem:faces-construction}, since the affine dimension of $Q_i$ does not depend on $Z_1^{(i)},Z_2^{(i)}$, if we choose $z_1^\ast,z_2^\ast \in \Delta_i^0$, Corollary~\ref{cor:relations-of-faces} yields \[f_\ell({K_n^\mu)}=C_i(\mathcal{F})+f_\ell(\conv(\widetilde{R}_n))+f_{\ell}(P_i(Z_1^{(i)},Z_2^{(i)})),\] where \[\widetilde{R_n}=(R_n \cup \{z_1^\ast,z_2^\ast\})\setminus\{Z_1^{(i)},Z_2^{(i)}\}.\] Observe that $\widetilde{R_n}$ has two fewer random points, with respect to $\mathcal{F}$, than $R_n$. Repeating the process, we get \[f_\ell(K_n^\mu)=C(\mathcal{F})+\sum_{i \in I} f_\ell(P_i(Z_1^{(i)},Z_2^{(i)})).\] Thus, as by Corollary~\ref{cor:independence-under-condition}, under $\mathcal{F}$, the pairs $(Z_1^{(i)},Z_2^{(i)})$ are independent and for every $i \in I$ \[Z_1^{(i)}\mid \mathcal{F}, \,\ Z_2^{(i)}\mid\mathcal{F}\sim\mu_{\Delta_i^0},\] applying the  law of total variance and Proposition~\ref{prop:partial-faces-lower-bound}, we obtain \begin{align*}\Var\bigl(f_\ell(K_n^\mu)\bigr)  
    \geq \mathbb{E}\bigl[\Var\bigl(V_\ell(K_n^\mu)\mid \mathcal{F}\bigr)\bigr]\gtrsim  \sum_{i=1}^{m}\mathbb{P}(A_i) \gtrsim_\mu  m \asymp r_n^{\frac{d-1}{2}}h_n^{-\frac{d-1}{2}},
\end{align*}
    where we also used Claim~\ref{claim:A-face-events-lower-bound} and
\eqref{eq:m-estimate}. 
To conclude, we need to lower-bound \[\min_{d+2 \leq n \leq\,\widetilde{n}_0} \Var\bigl(f_\ell(K_n^\mu)\bigl).\] The following claim is enough.
    \begin{claim}
        For every $n \geq d+2$, \[\Var\bigl(f_\ell(K_n^\mu)\bigl)>0.\]
    \end{claim}

    \begin{proof}
     Let the $d$-simplex $P_d=\conv\{0,e_1,\ldots,e_d\}$ and set \[F=\conv\{e_1,\ldots,e_d\}, \quad z=\frac{1}{\sqrt{d}}(1,\ldots,1).\]  Then $z$ lies beyond $F$ and beneath every other facet of $P_d$. Let also $B$ be a closed ball such that $B \subseteq {\rm int} (P_d)$. So we can find open neighborhoods \[U_0,U_1,\ldots,U_d,U_z\] of $0,e_1,\ldots,e_d$ and $z$ respectively such that, if $z_i \in U_k, \, k=0,\ldots,d$ and $\tilde{z} \in U_z$, then $\widetilde{P}_d=\conv\{z_0,\ldots,z_d\}$ is a $d$-simplex containing $B$ and $\tilde{z}$ lies beyond the facet $\widetilde{F}=\conv\{z_1,\ldots,z_d\}$ and beneath every other facet of $\widetilde{P}_d$. We can do so, because all the necessary properties are defined via strict inequalities, that involve continuous functions. Define first the event \[E_1=\left(\bigcap_{k=1}^{d+1}\big\{X_k \in U_{k-1}\}\right) \cap\left( \bigcap_{k=d+2}^{n} \big \{X_k \in B\}\right).\] Then $\mathbb{P}(E_1)>0$ and on $E_1$, $K_n^\mu$ is a $d$-simplex. Thus, \[f_\ell(K_n^\mu)=\binom{d+1}{\ell+1}.\] In addition, define the event \[E_2=\left(\bigcap_{k=1}^{d+1}\big\{X_k \in U_{k-1}\}\right) \cap  \big\{X_{d+2 }\in U_z\big\}\cap\left( \bigcap_{k=d+3}^{n} \big \{X_k \in B\}\right).\] Then $\mathbb{P}(E_2)>0$ and on $E_2$, $K_n^\mu$ is obtained from a $d$-simplex, by stacking a pyramid onto one of its facets. Therefore, \[f_\ell(K_n^\mu)=\binom{d+1}{\ell+1}+\tau_{d,\ell},\] where $\tau_{d,\ell} \geq 1$. This implies that $f_\ell(K_n^\mu)$ is not constant with positive probability.
    \end{proof}  
\end{proof}

\bigskip

\noindent {\bf Acknowledgements.} We would like to thank S.~Brazitikos for helpful discussions and A.~Giannopoulos for his constructive comments.

\small

\bigskip

\noindent\textbf{Keywords.}
random polytopes; intrinsic volumes; face numbers; variance lower bounds;
rotationally invariant log-concave measures.

\noindent\textbf{Mathematics Subject Classification (2020).}
Primary: 60D05; Secondary: 52A22, 52A23, 52B05.

\bigskip

\bigskip

\noindent \textsc{Minas \ Pafis}: Department of Mathematics, National and Kapodistrian University of Athens, Panepistimioupolis 157-84,
Athens, Greece.

\smallskip

\noindent \textit{E-mail:} \texttt{mipafis@math.uoa.gr}

\bigskip

\noindent \textsc{Christos \ Pandis}: Department of Mathematics and Applied Mathematics, University of Crete, Voutes Campus, 70013 Heraklion, Greece.

\smallskip

\noindent \textit{E-mail:} \texttt{chrpandis@gmail.com}

\end{document}